\documentclass[11pt,reqno]{amsart}
\usepackage{a4wide}
\usepackage{amsthm,amsfonts,amsmath,amscd,amssymb,epsfig}
\usepackage{hyperref}
\usepackage{color}
\usepackage{amsfonts,amssymb}
\usepackage{mathrsfs}

\usepackage{color} 

\definecolor{green}{rgb}{0,0.6,0}
\definecolor{blue}{rgb}{0,0,1}

\hypersetup{colorlinks,
linkcolor=blue,
filecolor=green,
citecolor=green}

\theoremstyle{plain}
\newcommand{\set}[2]{\left\{{#1}\mid{#2}\right\}}      
\newtheorem{theorem}{Theorem}[section]
\newtheorem{thm}[theorem]{Theorem}

\newtheorem{neu}{}[section]
\newtheorem{Cor}[neu]{Corollary}
\newtheorem*{Cor*}{Corollary}
\newtheorem{Thm}[neu]{Theorem}
\newtheorem*{Thm*}{Theorem}
\newtheorem{Prop}[neu]{Proposition}
\newtheorem*{Prop*}{Proposition}
\theoremstyle{definition}
\newtheorem{Lemma}[neu]{Lemma}
\newtheorem*{Rmk*}{Remark}
\newtheorem{Rmk}[neu]{Remark}
\newtheorem{Ex}[neu]{Example}
\newtheorem*{Ex*}{Example}

\newtheorem*{Qu*}{Question}

\newtheorem{Def}[neu]{Definition}

\theoremstyle{remark}

\theoremstyle{definition}

\newcommand{\p}{\partial}

\newcommand{\into}{\hookrightarrow}
\newcommand{\pf}{\longrightarrow}

\newcommand{\N}{{\mathbb{N}}}
\newcommand{\Z}{{\mathbb{Z}}}
\newcommand{\R}{{\mathbb{R}}}

\newcommand{\T}{{\mathbb{T}}}
\newcommand{\A}{{\SSS}}
\newcommand{\SSS}{{\mathbb{S}}}

\renewcommand{\H}{\mathrm{H}}

\newcommand{\Crit}{{\rm Crit}}

\newcommand{\PS}{\mathrm{PS}}

\renewcommand{\H}{\mathcal{H}}

\newcommand{\MM}{\mathcal{M}}
\newcommand{\x}{\times}

\newcommand{\beq}{\begin{equation}}
\newcommand{\beqn}{\begin{equation}\nonumber}
\newcommand{\eeq}{\end{equation}}

\newcommand{\bea}{\begin{equation}\begin{aligned}}
\newcommand{\bean}{\begin{equation}\begin{aligned}\nonumber}
\newcommand{\eea}{\end{aligned}\end{equation}}

\renewcommand{\theenumi}{\roman{enumi}}
\renewcommand{\labelenumi}{(\theenumi)}

\begin{document}
\title{Lectures on the free period Lagrangian action functional}

\author{Alberto Abbondandolo}
\address{Ruhr Universit\"at Bochum,
Fakult\"at f\"ur Mathematik,
Geb\"aude NA 4/33,
D-44801 Bochum,
Germany}
\email{alberto.abbondandolo@rub.de}

\maketitle


\let\thefootnote\relax\footnote{The present work is part of the author's activities within CAST, a Research Network Program of the European Science Foundation.}

\centerline{\em To Kazimierz G\c{e}ba on the occasion of his 80th birthday}

\begin{abstract}
In this expository article we study the question of the existence of periodic orbits of prescribed energy for classical Hamiltonian systems on compact configuration spaces. We use a variational approach, by studying how the behavior of the free period Lagrangian action functional changes when the energy crosses certain values, known as the Ma\~n\'e critical values.
\end{abstract} 

\tableofcontents

\section*{Introduction}

The main topic of this expository article is the question of the existence of periodic orbits of prescribed energy for classical Hamiltonian systems on compact configuration spaces. 

More precisely, we consider a connected closed manifold $M$ and a smooth {\em Tonelli Lagrangian} $L$ on the tangent bundle $TM$ of $M$: The Tonelli assumption means that $L$ is fiberwise uniformly convex and superlinear. It is a very natural assumption: For instance, it guarantees that the Legendre transform is well defined and produces a diffeomorphism between the tangent and the cotangent bundle of $M$, it allows to prove that every pair of points in $M$ is connected by a curve $\gamma$ which minimizes the Lagrangian action 
\[
\int_{t_0}^{t_1} L(\gamma(t),\gamma'(t)) \, dt
\]
and is a solution of the Euler-Lagrange equation associated to $L$, which is unique whenever the two points are sufficiently close to each other (see e.g. \cite{bgh98} or \cite{maz11}). Typical examples of Tonelli Lagrangians are {\em electromagnetic Lagrangians}, that is functions of the form
\[
L(x,v) = \frac{1}{2} |v|_x^2 + \theta(x)[v] - V(x), \qquad \forall (x,v)\in TM,
\]
where $|\cdot|_x$ denotes the norm associated to a Riemannian metric on $M$ (the kinetic energy), $\theta$ is a smooth one-form (the magnetic potential) and $V$ is a smooth function (the scalar potential) on $M$. 

The Euler-Lagrange equations associated to $L$ induce a smooth flow on $TM$ which preserves the energy function $E:TM \rightarrow \R$,
\[
E(x,v) := d_v L(x,v)[v] - L(x,v), \qquad \forall (x,v)\in TM.
\]
Given a number $\kappa\in [\min E,+\infty)$, the problem under considerations is the existence of a periodic orbit on the energy level $E^{-1}(\kappa)$.

Such a problem has been studied by several authors, for several classes of Tonelli Lagrangians and energy ranges, and by several techniques. For instance, in the case of Lagrangians of the form 
\[
L(x,v) = \frac{1}{2} |v|_x^2 - V(x),
\]
one can use the Maupertuis-Jacobi metric as in \cite{ben84} and reduce the problem to the existence of closed Riemannian geodesics either on $M$, if $\kappa$ is larger than the maximum of $V$, or on the domain $\{V\leq \kappa\}\subset M$, which is endowed with a metric which degenerates on the boundary $V^{-1}(\kappa)$, if $\kappa$ is smaller than the maximum of $V$. For a general Tonelli Lagrangian, the role of the maximum of $V$ is played by the {\em Ma\~{n}\'e critical values}. More precisely, important values of the energy are the numbers
\[
\min E \leq e_0(L) \leq c_u(L) \leq c_0(L),
\]
where $e_0(L)$ is the maximal critical value of $E$, $c_u(L)$ is minus the infimum of the mean Lagrangian action
\[
\frac{1}{T} \int_0^T L(\gamma(t)\, \gamma'(t))\, dt
\]
over all contractible closed curves $\gamma$, and  $c_0(L)$ is minus the infimum of the mean Lagrangian action over all null-homologous closed curves. In the case of electromagnetic Lagrangians, $e_0(L)=c_u(L)=c_0(L)$ when the magnetic potential $\theta$ vanishes, but the first and the second values are in general distinct when $\theta$ does not vanish ($c_u(L)$ and $c_0(L)$ can be distinct only when the fundamental group of $M$ is sufficiently non-abelian). The importance of $e_0(L)$ is clear, since it marks a change in the topology of $E^{-1}(\kappa)$: If $\kappa>e_0(L)$, then $E^{-1}(\kappa)$ is diffeomorphic to the unit tangent bundle of $M$, if $\kappa<e_0(L)$ then the projection of $E^{-1}(\kappa)$ to $M$ is not surjective anymore. The {\em lowest} Ma\~{n}\'e critical value $c_u(L)$ affects directly the behavior of the   
{\em free period Lagrangian action functional}
\[
\SSS_{\kappa}(\gamma) := \int_0^T \Bigl( L\bigl(\gamma(t),\gamma'(t)\bigr) + \kappa \Bigr) \, dt, \qquad \gamma: \R/T\Z \rightarrow M.
\]
The critical points of this functional, whose domain is a suitable space of closed curves $\gamma$ in $M$ of arbitrary period $T$, are exactly the closed orbits of energy $\kappa$. The functional $\SSS_{\kappa}$ is bounded from below on every connected component of the free loop space whenever $\kappa\geq c_u(L)$, and it is unbounded from below on every such connected component when $\kappa< c_u(L)$.
The {\em strict} Ma\~{n}\'e critical value $c_0(L)$ is not directly related to the topology of $\SSS_{\kappa}$, but it has dynamical and geometric significance: For $\kappa>c_0(L)$ the energy surface $E^{-1}(\kappa)$ is of {\em restricted contact type}, and the Euler-Lagrangian flow on it is conjugated, up to a time reparametrization, to a {\em Finsler geodesic flow} on $M$, whereas both facts are in general false for $\kappa\leq c_0(L)$. Furthermore, the Ma\~{n}\'e critical values are related to compactness properties of the functional $\SSS_{\kappa}$, such as the {\em Palais-Smale condition}. 

By exploiting these facts, the free period action functional $\SSS_{\kappa}$ can be effectively used as a variational principle for our problem and allows to prove various results, which we summarize into the following theorem.

\begin{Thm*}
Let $L$ be a Tonelli Lagrangian on the tangent bundle of the closed manifold $M$.
\begin{enumerate}
\item If $\kappa>c_u(L)$ and $M$ is not simply connected, then the energy level $E^{-1}(\kappa)$ has a $\SSS_{\kappa}$-minimizing periodic orbit in each non-trivial homotopy class of the free loop space of $M$.
\item If $\kappa>c_u(L)$ and $M$ is simply connected, then the energy level $E^{-1}(\kappa)$ has a periodic orbit with positive  action $\SSS_{\kappa}$.
\item For almost every $\kappa\in (\min E,c_u(L))$  the energy level $E^{-1}(\kappa)$ has a periodic orbit with positive  action $\SSS_{\kappa}$.
\item If the energy level $E^{-1}(\kappa)$ is stable then $E^{-1}(\kappa)$ has a periodic orbit.
\end{enumerate}
\end{Thm*}

Notice that in (iii) only existence for {\em almost every} energy level in the interval $(\min E,c_u(L))$ (in the sense of Lebesgue measure) is stated: existence for {\em all} energy levels in this range is still unknown, although no counterexamples have been found so far. This issue is related to the fact that the Palais-Smale condition does not hold anymore below $c_u(L)$. The {\em stability condition} which is assumed in (iv) is a weaker form of the contact type condition.

The above theorem was first proved in this form by G.~Contreras \cite{con06} (assuming contact type instead of stable in (iv)), building on previous geometric ideas of I.~A.~Taimanov \cite{tai83,tai92b}. Contreras' long paper \cite{con06} contains many other beautiful results, such as the study of the invariant probability measures which one obtains as limits of Palais-Smale sequences which do not converge in the free loop space. This article is meant to be a gentle introduction, including some technical simplifications, to the part of \cite{con06} which concerns periodic orbits.

Unlike in a typical survey article, we are more concerned with detailed proofs, which we try to make accessible to a large audience including students, than with a systematic overview of the literature, for which we refer to the beautiful survey of Taimanov \cite{tai92b}, to \cite{cmp04} and to the already cited \cite{con06}. In particular, we start by proving well known abstract results, such as the mountain pass theorem, the general minimax principle, and the construction of the structure of infinite dimensional Hilbert manifold on the space of closed loops on $M$ of Sobolev class $W^{1,2}$ (Sections \ref{mpsec} and \ref{hmlsec}). In Section \ref{sec3} we introduce the already mentioned free period action functional $\SSS_{\kappa}$, which plays a fundamental role in this article and gives it its title: Unlike in \cite{con06}, we use it also to get existence of closed orbits for energies below $e_0(L)$. The Ma\~{n}\'e critical values which are relavant for this article are introduced in Section \ref{mcv}, together with some of their characterizations and with the discussion of two geometric properties of an energy level, namely the contact type and the stability condition. The analysis of Palais-Smale sequences is carried out in Section \ref{pss}, and in Section \ref{pohe} we prove statements (i) and (ii) of the above theorem. The topology of the free period action functional for $\kappa<c_u(L)$ is studied in  Section \ref{tfpafle}, and in Section \ref{pole} we finally prove statements (iii) and (iv), using {\em Struwe's monotonicity argument}, together with a weaker version of (iii), using an alternative argument.

\medskip

\paragraph{\bf Acknowledgments} This expository article is the outcome of two series of lectures that the author gave at two summer schools at the Korea Institute for Advanced Study of Seul in 2010 and at the Universit\'e de Neuch\^atel in 2011, respectively. I am grateful to Urs Frauenfelder and Felix Schlenk for organizing these two events, and to Gabriel Paternain, who was also a speaker at the first school, for many fruitful discussions. I would like to thank Jungsoo Kang, who participated to the first school, organized the material and typed a first version of the notes which eventually became this article. I would like to thank also Luca Asselle, who participated to the second school and suggested Lemma \ref{PST0} below, which allows to avoid extra technicalities and the detailed analysis of Palais-Smale sequences with infinitesimal periods which was present in the previous notes. 

\numberwithin{equation}{section}

\section{The minimax principle}
\label{mpsec}

\paragraph{\bf The mountain pass theorem} Let $H$ be a real Hilbert space and let $f$ be a continuously differentiable real function on $H$. The symbol $\Crit f$ denotes the set of critical points of $f$.
We assume that a certain open sublevel $\{f<a\}$ is
not connected, say $\{f<a\}=A\cup B$, with $A$ and $B$ disjoint non-empty open sets. We may think of $A$ and $B$ as two valleys, and consider the set of paths going from one valley to the other one, that is the set
$$
\Gamma:=\{\textrm{curves in $H$ with one end in $A$ and the other in $B$}\}.
$$
We can define the minimax value of $f$ on $\Gamma$ as
$$
c:=\inf_{\gamma\in\Gamma}\max_{x\in\gamma} f(x),
$$
and we notice that $a\leq c < +\infty$, because $\Gamma$ is non empty and each of its elements intersects the set $H\setminus (A\cup B) = \{f\geq a\}$.
One would expect this mountain pass level $c$ to be a critical value of $f$.
The next simple example shows that this is not always the case.

\begin{Ex}
Consider the smooth function $f$ on $\R^2$ defined by
\[
f(x,y)=e^x-y^2.
\]
Then $\{f<0\}$ has two connected components, $c=0$, but $f$ has no critical points. The problem here is that the critical point is pushed to infinity: Indeed, $f(-n,0)=e^{-n}$ converges to the mountain pass level $c=0$ and $df(-n,0)=e^{-n} dx$ tends to zero.
\end{Ex}

This example suggests the following definition.

\begin{Def}
A sequence $(x_n)_{n\in\N}\subset H$ is called a {\em Palais-Smale sequence} at level $c$ ($(\PS)_c$ for short) if
$$
\lim_{n\to\infty}f(x_n)=c\quad\mbox{and}\quad \lim_{n\to\infty}df(x_n)=0.
$$
The function $f$ is said to satisfy $(\PS)_c$ if all $(\PS)_c$ sequences are compact. It is said to satisfy $(\PS)$ if it satisfies $(\PS)_c$ for every $c\in \R$.
\end{Def}

Notice that limiting points of $(\PS)_c$ sequences are critical points at level $c$. We can now state the celebrated mountain pass theorem of Ambrosetti and Rabinowitz \cite{ar73} in the following form:

\begin{Thm}[Mountain Pass Theorem]
\label{mp}
Let $f\in C^{1,1}(H)$ be such that  $\{f<a\}$ is not connected and let $c$ be defined as above. Then $f$ admits a $(\PS)_c$ sequence. In particular, if $f$ satisfies $(\PS)_c$, then $c$ is a critical value.
\end{Thm}

Here $C^{1,1}$ denotes the set of functions whose differential is locally Lipschitz-continuous.

\begin{proof}
By contradiction, suppose that there exists $\epsilon>0$ such that $||df|| \geq\epsilon$ on the set $\{|f-c|\leq\epsilon\}$. We denote by $\nabla f$ the gradient of $f$ and we assume for sake of simplicity that the locally Lipschitz vector field $-\nabla f$ is positively complete, meaning that its flow $\phi$, that is the solution of
\[
\left\{
\begin{aligned}
&\frac{\p}{\p t}\phi_t (u)=-\nabla f\bigl(\phi_t(u)\bigr),\\[1ex]
&\phi_0 (u)=u,
\end{aligned}
\;\;\right.
\]
is defined for every $t\geq 0$ and every $u\in H$. This holds, for instance, if $\nabla f$ is globally Lipschitz (in this case the flow of $-\nabla f$ is defined on the whole $\R\times H$). See Remark \ref{noncomp} below for a hint on how to remove this extra assumption. Notice that
\begin{equation}
\label{decr}
\frac{d}{dt} f\bigl(\phi_t(u)\bigr) = df\bigl( \phi_t(u) \bigr) \bigl[-\nabla f\bigl(\phi_t(u)\bigr) \bigr] = - \bigl\|df\bigl(\phi_t(u)\bigr)\bigr\|^2,
\end{equation}
so the function $t\mapsto f(\phi_t(u))$ is decreasing. If $|f(\phi_t(u))-c|\leq\epsilon$ for all $t\in[0,T]$, we have
\[
2\epsilon \geq f(u)-f(\phi_T(u)) =-\int_0^T\frac{d}{dt}f(\phi_t(u))dt
=\int_0^T \bigl\|d f\bigl(\phi_t(u)\bigr)\bigr\|^2dt
\geq\epsilon^2T,
\]
from which we conclude that $T\leq 2/\epsilon$. Choose $\gamma\in\Gamma$ such that $\max_\gamma f\leq c+\epsilon$ and set
$$
\tilde\gamma=\phi_T(\gamma), \qquad\textrm{for some } T>\frac{2}{\epsilon}.
$$
The fact that $f$ decreases along the orbits of $\phi$ implies that $\tilde\gamma$ belongs to $\Gamma$. Since $f\leq c+\epsilon$ on $\gamma$, any $x\in\gamma$ satisfies either (i) $|f(x)-c|\leq\epsilon$ or (ii) $f(x)<c-\epsilon$. Let $x\in \gamma$.
If (i) holds, then $f(\phi_T(x))<c-\epsilon$ because $T>2/\epsilon$. If (ii) holds, then $f(\phi_T(x))<c-\epsilon$ because $f$ decreases along the orbits of $\phi$. Therefore we conclude that $\tilde\gamma\subset\{f< c-\epsilon\}$, which contradicts the definition of $c$.
\end{proof}

\begin{Rmk}
\label{noncomp}
If the vector field $-\nabla f$ is not positively complete, we can replace it by the complete one
 $-\nabla f/\sqrt{||\nabla f||^2+1}$. The above proof goes through with minor adjustments.
\end{Rmk}

\begin{Rmk}
The mountain pass theorem holds also for $f\in C^{1,1}(\MM)$ where $(\MM,g)$ is a Hilbert manifold equipped with a complete Riemannian metric $g$.
In this case, $(x_n)_{n\in\N}\subset\MM$ is said to be a $(\PS)_c$ sequence if $\lim_{n\to\infty}f(x_n)=c$ and $\lim_{n\to\infty}||df(x_n)||=0$, where $\|\cdot\|$ denotes the dual norm induced by $g$. Notice that the (PS) condition and the completeness of $g$ are somehow antagonist requirements: One may always achieve the completeness of an arbitrary Riemannian metric $g$ by multiplying it by a positive function which diverges at infinity (such an operation reduces the set of the Cauchy sequences), while the (PS) condition could be achieved by multiplying $g$ by a positive function which is infinitesimal at infinity (since the dual norm is multiplied by the inverse of this function, this operation reduces the set of the (PS) sequences).
\end{Rmk}

\begin{Rmk}
The mountain pass theorem holds also if $f$ is just continuously differentiable. In this case, its negative gradient vector field is just continuous and may not induce a continuous flow. In order to prove the above theorem, one needs to construct a locally Lipschitz pseudo-gradient vector field for $f$, see for instance \cite[Lemma 3.2]{str00}. The same construction allows to prove the mountain pass theorem for continuously differentiable functions on Banach spaces, or more generally on Banach manifolds.
\end{Rmk}

\begin{Rmk}
When dealing with functions on manifolds, it is sometimes useful to have a formulation of the mountain pass theorem which does not involve the choice of a metric. Here is such a formulation. Assume that $f$ is a continuously differentiable function on a Hilbert manifold $\MM$ and that $V$ is a positively complete locally Lipschitz vector field such that $df[V]<0$ on $\MM \setminus\Crit f$. Then the mountain pass theorem holds, provided that we define $(x_n)_{n\in\N}\subset \MM$ to be a $(\PS)_c$ sequence if $f(x_n)$ tends to c and $df(x_n)[V(x_n)]$ is infinitesimal. Now the antagonism is between this form of the (PS) condition and the positive completeness of $V$.
\end{Rmk}

\paragraph{\bf The general minimax principle} In the proof of Theorem \ref{mp} we have not used the fact that $\Gamma$ is a set of curves, but rather that $\Gamma$ is  positively invariant with respect to the negative gradient flow $\phi$ of $f$, meaning that  $\phi_t(\gamma)\in\Gamma$ for all $\gamma\in\Gamma$ and $t\geq 0$. Here $\phi$ is either the flow of $-\nabla f$, when this vector field is positively complete, or the flow of some conformally equivalent positively complete vector field, such as
$-\nabla f/\sqrt{\|\nabla f\|^2+1}$, in the general case.
This simple observation leads to the following powerful generalization of the mountain pass theorem.

\begin{Thm}[\bf General Minimax Principle]
\label{thm:finite c induces PS}
Let $f$ be a $C^{1,1}$ function on the complete Riemannian Hilbert manifold $(\MM,g)$ and let $\Gamma$ be a set of subsets of $\MM$ which is positively invariant with respect to the negative gradient flow of $f$. If the number
\[
c= \inf_{\gamma\in\Gamma}\sup_\gamma f
\]
is finite, then $f$ admits a $(\PS)_c$ sequence. In particular, if $f$ satisfies $(\PS)_c$, then $c$ is a critical value.
\end{Thm}

The proof is a straightforward modification of the proof of Theorem \ref{mp}.

\begin{Ex}
Let $f\in C^{1,1}(H)$, where $H$ is a Hilbert space. If $\pi_k(\{f<a\})\ne 0$ for some $k\geq 0$ and $f$ satisfies $(\PS)$,  then $f$ has a critical point. Indeed, we can consider the set
\[
\begin{split}
\Gamma := \bigl\{ z(\overline{B}^{k+1}) \; \Big| \; & z: (\overline{B}^{k+1},\partial B^{k+1}) \rightarrow (H,\{f<a\}) \mbox{ continuous map such that }\\  & [z|_{\partial B^{k+1}}] \neq 0 \mbox{ in } \pi_k(\{f<a\})\Bigr\},
\end{split}
\]
where $B^{k+1}$ denotes the unit open ball of dimension $k+1$. By applying Theorem \ref{thm:finite c induces PS} with such a $\Gamma$ we get the existence of a critical point at level $c\geq a$. The case $k=0$ is precisely the Mountain Pass Theorem \ref{mp}.
\end{Ex}

\begin{Rmk}
\label{mini}
If $\Gamma$ is the class of all one-point sets in $\MM$, then $c$ is the infimum of $f$. Therefore, the general minimax principle has as a particular case the following existing result for minimizers: Assume that $f\in C^{1,1}(\MM)$ is bounded from below, has complete sublevels and satisfies $(\PS)_c$ at the level $c=\inf f$; then $f$ has a minimizer.
\end{Rmk}

\begin{Rmk}
\label{trunc}
It is sometimes useful to replace the negative gradient flow by a flow which fixes a certain sublevel of $f$. Let $\rho:\R\pf\R^+$ be a smooth bounded function such that $\rho=0$ on $(-\infty,b]$ and $\rho>0$ on $(b,+\infty)$. Then we consider the vector field $V=-\rho(f)\cdot\nabla f$ (or $V= - \rho(f) \nabla f/\sqrt{\|\nabla f\|^2+1}$ in the non-positively complete case) and denote its flow by $\phi$. It is a negative gradient flow truncated below level $b$: The function $t\mapsto f(\phi_t(u))$ is constant if $u\in \Crit f \cup \{f\leq b\}$ and it is strictly decreasing otherwise.
If $\Gamma$ is positively invariant with respect to this negative gradient flow truncated below level $b$ and the minimax value $c$ is strictly larger than $b$, then $f$ has a $(\PS)_c$ sequence.
\end{Rmk}

\section{A Hilbert manifold of loops}
\label{hmlsec}

Let $(M,g)$ be a closed Riemannian manifold of dimension $n$ and consider the Sobolev space of loops
$$
W^{1,2}(\T,M):=\Big\{x:\T\pf M\,\Big|\,x\textrm{ is absolutely continuous and } \int_\T|x'(s)|^2_{x(s)}ds<\infty\Big\},
$$
where $\T:=\R/\Z$ and $|\cdot|_{\cdot}$ denotes the norm induced by $g$. This set of loops is clearly independent from the choice of the Riemannian metric $g$.

\medskip

\paragraph{\bf The smooth structure of $\mathbf{W^{1,2}(T,M)}$} Let us recall the construction of the smooth Hilbert manifold structure on $W^{1,2}(\T,M)$.
Fix $x_0\in C^\infty(\T,M)$. Assume for simplicity that $x_0$ preserves the orientation, so that $x_0^*(TM)$ has a trivialization
$$
\Phi: \T\x\R^n\pf x_0^*(TM).
$$
Let $B_r$ be the open ball of radius $r$ about $0$ in $\R^n$. Consider a smooth map
\[
\varphi:\T\x B_r \pf M,
\]
such that $\varphi(t,0)=x_0(t)$ and $\varphi(t,\cdot)$ is a diffeomorphism onto an open subset in $M$, for every $t\in \T$. For instance, the map
\[
\varphi(t,\xi)= \exp_{x_0(t)}\big(\Phi(t,\xi)\big),
\]
satisfies the above requirements if $r$ is small enough.

The map $\varphi$ induces the following parameterization:
\begin{equation}
\label{eq:chart of the loop space}
\varphi_*: W^{1,2}(\T,B_r)\pf W^{1,2}(\T,M), \quad
\zeta\mapsto\varphi\big(\cdot,\zeta(\cdot)\big),
\end{equation}
where $W^{1,2}(\T,B_r)$ denotes the open subset of the Hilbert space $W^{1,2}(\T,\R^n)$ which consists of loops taking values into $B_r$. The collection of all these parameterizations, for every $x_0\in C^{\infty}(\T,M)$ and every $\varphi$ as above, defines a smooth atlas for $W^{1,2}(\T,M)$, which is then a smooth manifold modeled on the Hilbert space $W^{1,2}(\T,\R^n)$. Indeed, the smoothness of the transition maps is an immediate consequence of the chain rule.
It is worth noticing that the image of the parameterization $\varphi_*$ is $C^0$-open.

\begin{Rmk}
If $x_0$ is not orientation preserving, the natural model for the connected component of $W^{1,2}(\T,M)$ which contains $x_0$ is the space of $W^{1,2}$ sections of the vector bundle $x_0^*(TM)$. Alternatively, one can define a manifold structure on $W^{1,2}([0,1],M)$ without encountering topological problems, and then see $W^{1,2}(\T,M)$ as the inverse image of the diagonal of $M \times M$ by the smooth submersion 
\[
W^{1,2}([0,1],M) \rightarrow M \times M, \qquad x \mapsto (x(0),x(1)).
\]
\end{Rmk}

The tangent space of $W^{1,2}(\T,M)$ at $x$ is naturally identified with the space of $W^{1,2}$ sections of $x^*(TM)$. Therefore, we can define a Riemannian metric on $W^{1,2}(\T,M)$ by setting
\begin{equation}
\label{metr}
\langle \xi,\eta\rangle_x:=\int_{\T} \Bigl( g(\xi,\eta)+g(\nabla_t\xi,\nabla_t\eta) \Bigr) \,dt, \quad \forall \xi,\eta\in T_x W^{1,2}(\T,M),
\end{equation}
where $\nabla_t$ denotes the Levi-Civita covariant derivative along $x$. The distance induced by this Riemannian metric is compatible with the topology of $W^{1,2}(\T,M)$.
The fact that $M$ is compact implies that this metric on $W^{1,2}(\T,M)$ is complete (more generally, this metric is complete whenever $g$ is complete).

The gradient of functionals on $W^{1,2}(\T,M)$ is the one which is associated to such a Riemannian metric.

\begin{Rmk}
If $\varphi$ is the restriction of a smooth map $B_{r'}\times \T\rightarrow M$ with the same properties, for some $r'>r$, then the parameterization $\varphi_*$ is bi-Lipschitz.
\end{Rmk}

See e.g. \cite{kli82} for more details on the Hilbert manifold structure of $W^{1,2}(\T,M)$.

\medskip

\paragraph{\bf The homotopy type of $\mathbf{W^{1,2}(T,M)}$} The inclusions
$$
C^\infty(\T,M)\into W^{1,2}(\T,M)\into C(\T,M)
$$
are dense homotopy equivalences. These facts can be proved by embedding $M$ into a Euclidean space $\R^N$, by regularizing the loops $x:\T\rightarrow M\subset \R^N$ by convolution, and by projecting the regularized loop back to $M$ using the tubular neighborhood theorem. In particular, the connected components of $W^{1,2}(\T,M)$ are in one-to-one correspondence with the conjugacy classes of
$\pi_1(M)$. See, e.g., \cite[Chapter 10]{lee03} for more details.

\section{The free period action functional}
\label{sec3}

\paragraph{\bf Tonelli Lagrangians} Let $M$ be a connected closed manifold.
A function $L\in C^\infty(TM)$ is called a {\em Tonelli Lagrangian} if:
\begin{enumerate}
\item $L$ is fiberwise uniformly convex, i.e.\ $d_{vv}L(x,v)>0$ for every $(x,v)\in TM$, where $d_{vv}L$ denotes the fiberwise second differential of $L$;
\item $L$ has superlinear growth on each fiber, i.e.
\[
\lim_{|v|\to +\infty}\frac{L(x,v)}{|v|_x}=+\infty.
\]
\end{enumerate}
The main example of Tonelli Lagrangians is given by the {\em electromagnetic Lagrangians}, that is functions of the form
\begin{equation}
\label{elmag}
L(x,v)=\frac{1}{2}|v|^2_x+\theta(x)[v]-V(x),
\end{equation}
where $|\cdot|_x$ denotes the norm associated to a Riemannian metric (the kinetic energy), $\theta$ is a smooth one-form (the magnetic potential) and $V$ is a smooth function (the scalar potential) on $M$. We shall omit the subscript $x$ in $|\cdot|_x$ when the point $x$ is clear from the context.
The Tonelli assumptions imply that the Euler-Lagrange equation, which in local coordinates can be written as
\begin{equation}
\label{EL}
\frac{d}{dt}\bigl( \p_v L (\gamma(t),\gamma'(t)) \bigr)= \p_x L (\gamma(t),\gamma'(t)),
\end{equation}
is well-posed and defines a smooth flow on $TM$. This flow preserves the energy
\[
E:TM\rightarrow \R, \quad
E(x,v):=d_v L(x,v)[v]-L(x,v),
\]
where $d_v$ denotes the fiberwise differential.
When $L$ has the form (\ref{elmag}), then
\begin{equation}
\label{ene}
E(x,v) = \frac{1}{2} |v|^2 + V(x).
\end{equation}
In general, the energy function of a Tonelli Lagrangian satisfies the following properties:
\begin{itemize}
\item[(i)] $E$ is fiberwise uniformly convex and superlinear.
\item[(ii)] For any $x\in M$, the restriction of $E$ to $T_x M$ achieves its minimum at $v=0$.
\item[(iii)] The point $(\bar x,0)$ is singular for the Euler-Lagrange flow if and only if $(\bar x,0)$ is a critical point of $E$.
\end{itemize}

We are interested in proving the existence of periodic orbits on a given energy level $E^{-1}(\kappa)$. Since such an energy level is compact, up to the modification of $L$ far away from it, we may assume that the Tonelli Lagrangian $L(x,v)$ is electromagnetic for $|v|$ large. In particular, we have the inequalities
\begin{eqnarray}
\label{boundsonL}
L(x,v) \geq L_0 |v|^2 - L_1, \qquad & \forall (x,v)\in TM,\\
\label{bd2L}
d^2_{vv} L(x,v)[u,u] \geq 2 L_0 |u|^2, \qquad &\forall (x,v)\in TM, \; u\in T_x M,
\end{eqnarray}
for some numbers $L_0>0$ and $L_1\in \R$. Moreover, $E$ has the form (\ref{ene}) for $|v|$ large.

\medskip

\paragraph{\bf The free period action functional} We would like to study the Lagrangian action on the space of closed curves of arbitrary period. The latter space can be given a manifold structure by reparametrizing each curve on $\T$ and by keeping track of its period as a second variable:
Let $\gamma :\R/{T\Z}\pf M$ be an absolutely continuous $T$-periodic curve and define $x:\T \rightarrow M$ as $x(s):=\gamma(sT)$. The closed curve $\gamma$ is identified with the pair $(x,T)$.
The action of $\gamma$ on the time interval $[0,T]$ is the number
\[
\int_0^TL\bigl(\gamma(t),\gamma'(t)\bigr)\, dt= T\int_{\T} L\bigl(x(s),x'(s)/{T}\bigr)\, ds.
\]
Fix a real number $\kappa$, the value of the energy for which we would like to find periodic solutions. Consider the {\em free period action functional} corresponding to the energy $\kappa$
\[
\SSS_{\kappa}(\gamma) =
\SSS_{\kappa} (x,T) := T\int_{\T} \Bigl( L\bigl(x(s),x'(s)/T\bigr) + \kappa \Bigr)\,ds = \int_0^T \Bigl( L\bigl( \gamma(t),\gamma'(t)\bigr)+\kappa\Bigr)\, dt.
\]
The fact that $L$ is electromagnetic outside a compact subset of $TM$ implies that $\SSS_{\kappa}(x,T)$ is well-defined when $x\in W^{1,2}(\T,M)$. Therefore, we obtain a functional
\[
\SSS_{\kappa} : W^{1,2}(\T,M) \times (0,+\infty) \rightarrow \R.
\]
The Hilbert manifold $W^{1,2}(\T,M) \times (0,+\infty)$ is denoted by $\MM$.

\begin{Lemma}(Regularity properties of $\SSS_{\kappa}$)
\begin{itemize}
\item[(i)] $\SSS_{\kappa}$ is in $C^{1,1}(\MM)$ and is twice Gateaux differentiable at every point.
\item[(ii)] $\SSS_{\kappa}$ is twice Fr\'ech\'{e}t differentiable at every point if and only if $L$ is electromagnetic on the whole $TM$. In this case, $\SSS_{\kappa}$ is actually smooth on $\MM$.
\end{itemize}
\end{Lemma}

See e.g.\ \cite{as09b} for a detailed proof.

If $d_x$ denotes the horizontal differential with respect to some horizontal-vertical splitting of $TTM$, the differential of $\SSS_{\kappa}$  with respect to the first variable at some $(x,T)\in \MM$ has the form
\begin{equation}
\label{diff1}
\begin{split}
d\SSS_{\kappa} (x,T) \bigl[(\xi,0)\bigr] & = T \int_0^1 \Bigl( d_x L\bigl( x, x'/T \bigr) [ \xi ] + d_v L\bigl( x, x'/T \bigr) \bigl[ \nabla_s \xi/T \bigr] \Bigr)\, ds \\ & = \int_0^T \Bigl( d_x L \bigl(\gamma,\gamma'\bigr) [\zeta] +   d_v L \bigl(\gamma,\gamma'\bigr) [\nabla_t \zeta] \Bigr)\, dt,
\end{split}\end{equation}
where $\xi\in T_x W^{1,2}(\T,M)$, $\gamma(t)= x(t/T)$ and $\zeta(t):=\xi(t/T)$. 
Let $(x,T)$ be a critical point of $\SSS_{\kappa}$. The above formula and an integration by parts imply that $\gamma$ is a $T$-periodic solution of (\ref{EL}). Moreover
\begin{equation}
\label{eq:diff A w.r.t. T}
\begin{split}
\frac{\partial \SSS_{\kappa}}{\partial T}  (x,T) &= \int_{\T} \Bigl( L\bigl( x(s),x'(s)/T\bigr) + \kappa + T \, d_v L\bigl( x(s),x'(s)/T\bigr) \bigl[-x'(s)/T^2\bigr] \Bigr)\, ds \\ &= \int_{\T} \Bigl( \kappa - E\bigl( x(s),x'(s)/T\bigr) \Bigr)\, ds = \frac{1}{T}\int_0^T \Bigl( \kappa - E\bigl( \gamma(t),\gamma'(t)\bigr) \Bigr) \, dt.
\end{split}
\end{equation}
Together with the fact that $E$ is constant along the orbits of the Euler-Lagrange flow, the above identity shows that the $T$-periodic orbit $\gamma$ belongs to the energy levek $E^{-1}(\kappa)$. We conclude that $(x,T)$ is a critical point of $\SSS_{\kappa}$ on $\MM$ if and only if $\gamma(t):= x(t/T)$ is a $T$-periodic orbit of energy $\kappa$ ($T$ is not necessarily the minimal period).

\medskip

\paragraph{\bf Behavior of $\SSS_{\kappa}$ for $\mathbf{T\to 0}$} The Hilbert manifold $\MM=W^{1,2}(\T,M)\x (0,+\infty)$ is endowed with the product Riemannian structure of (\ref{metr}) and the Euclidean metric of $(0,+\infty) \subset \R$. As such, it is not complete, the non-converging Cauchy sequences being the sequences $(x_h,T_h)$ with $x_h \rightarrow x\in W^{1,2}(\T,M)$ and $T_h \rightarrow 0$. Therefore, we need to understand the behavior of $\SSS_{\kappa}$ on such sequences.

We decompose $\MM$ as $\MM = \MM^{\mathrm{contr}} \sqcup \MM^{\mathrm{noncontr}}$, where $\MM^{\mathrm{contr}}$ denotes the connected component consisting of contractible loops.

\begin{Lemma}
\label{Lem:3}
\begin{enumerate}
\item On $\MM^\mathrm{noncontr}$ the sublevels $\{\SSS_{\kappa}\leq c\}$ are complete. More precisely, if $(x_h,T_h)\in \MM^\mathrm{noncontr}$ and $T_h \rightarrow 0$, then $\SSS_{\kappa}(x_h,T_h)\rightarrow +\infty$.
\item If $(x_h,T_h)\in\MM^\mathrm{contr}$ and $T_h\to 0$, then $\liminf_h\SSS_{\kappa}(x_h,T_h)\geq0$.
\end{enumerate}
\end{Lemma}

\begin{proof}
By (\ref{boundsonL}), we have the chain of inequalities
\begin{equation}
\label{lowbdonA}
\begin{split}
\SSS_{\kappa} (x,T) &= T \int_{\T} \Bigl( L\bigl(x,x'/T\bigr) + \kappa \Big)\, ds \geq T \int_{\T} \Bigl( L_0 \frac{|x'|^2}{T^2} -L_1 + \kappa \Bigr)\, ds
\\ &= \frac{L_0}{T} \int_{\T} |x'|^2\, ds - (L_1-\kappa) T
\geq \frac{L_0}{T} \ell(x)^2 - (L_1 - \kappa)T,
\end{split} \end{equation}
where $\ell(x)$ denotes the length of the loop $x$. The length of the non-contractible loops in $M$ is bounded away from zero. Therefore, the estimate (\ref{lowbdonA}) implies statement (i). Statement (ii) is also an immediate consequence of (\ref{lowbdonA}).
\end{proof}

Since $\mathcal{M}$ is not complete, we cannot expect the vector field $-\nabla \SSS_{\kappa}$ to be positively complete. However, the only sources of non-completeness is the second component approaching zero. The next result says that this may happen only at level zero:

\begin{Lemma}\label{Lem:4}
Let $(x,T):[0,\sigma^*)\pf\MM^\mathrm{contr}$, $0<\sigma^*<\infty$, be a flow line of $-\nabla \SSS_{\kappa}$ such that
\[
\liminf_{\sigma\to\sigma^*}T(\sigma)=0.
\]
Then
\[
\lim_{\sigma\to\sigma^*}\SSS_{\kappa} \big(x(\sigma),T(\sigma)\big)=0.
\]
\end{Lemma}

\begin{proof}
Since both $E$ and $L$ are quadratic in $v$ for $|v|$ large, we have the estimate
\[
E(x,v) \geq C_0 \, L(x,v) - C_1,
\]
for some $C_0>0$ and $C_1\in \R$. From (\ref{eq:diff A w.r.t. T}) we obtain the inequality
\begin{equation*}
\begin{split}
\frac{\partial \SSS_{\kappa}}{\partial T} ( x,T) &= \frac{1}{T} \int_0^T \bigl( \kappa - E(\gamma,\gamma') \bigr)\, dt
\leq \frac{1}{T} \int_0^T \bigl( \kappa - C_0 \,L(\gamma,\gamma') + C_1 \bigr)\, dt \\
&= \kappa + C_1 - \frac{C_0}{T} \int_0^T \bigl( L(\gamma,\gamma')+\kappa \bigr)\, dt + C_0 \kappa = (C_0+1) \kappa + C_1 - \frac{C_0}{T} \SSS_{\kappa} (x,T),
\end{split} \end{equation*}
which can be rewritten as
\begin{equation}
\label{boundsondAdt}
\SSS_{\kappa} (x,T) \leq \frac{T}{C_0} \Bigl( C - \frac{\partial \SSS_{\kappa}}{\partial T}  (x,T) \Bigr),
\end{equation}
for a suitable constant $C$.
By the assumption, there exists an increasing sequence $(\sigma_h)$ which converges to $\sigma^*$ and satisfies $T'(\sigma_h)\leq 0$ and
$T(\sigma_h)\rightarrow 0$. Since $\sigma\mapsto (x(\sigma),T(\sigma))$ is a flow line of $-\nabla \SSS_{\kappa}$,
\[
0 \geq T'(\sigma_h) = - \frac{\partial \SSS_{\kappa}}{\partial T} \bigl(x(\sigma_h),T(\sigma_h) \bigr),
\]
and by (\ref{boundsondAdt}) we have
\[
\SSS_{\kappa} \bigl(x(\sigma_h),T(\sigma_h) \bigr) \leq \frac{T(\sigma_h)}{C_0} \Bigl( C - \frac{\partial \SSS_{\kappa}}{\partial T}  \bigl(x(\sigma_h),T(\sigma_h) \bigr) \Bigr) \leq \frac{C}{C_0} T(\sigma_h).
\]
Since $T(\sigma_h)$ is infinitesimal, we obtain
\[
\limsup_{h\rightarrow \infty} \SSS_{k} \bigl(x(\sigma_h),T(\sigma_h) \bigr) \leq 0.
\]
Together with statement (ii) of Lemma \ref{Lem:3} and the monotonicity of the function $\sigma \longmapsto \SSS_{\kappa}(x(\sigma),T(\sigma))$,
this concludes the proof.
\end{proof}

\section{Ma\~n{\'e} critical values, contact type and stability conditions}
\label{mcv}

\paragraph{\bf The critical values} The following numbers should be interpreted as energy levels and mark important dynamical and geometric changes for the Euler-Lagrange flow induced by the Tonelli Lagrangian $L$:
\begin{equation*}
\begin{split}
c_0(L)&:=\inf\{\kappa\in\R \, | \,\SSS_{\kappa}(x,T)\geq 0 \; \forall (x,T)\in \MM \mbox{ with }x\mbox{ homologous to zero} \} \\
&= - \inf \Bigl\{ \frac{1}{T} \int_0^T L(\gamma(t),\gamma'(t)) \, dt \, \Big| \, \gamma\in C^{\infty}(\R/T\Z,M) \mbox{ homologous to zero, } T>0\Bigr\}, \\
c_u(L)&:=\inf\{\kappa\in\R \, | \,\SSS_{\kappa}(x,T)\geq 0 \; \forall (x,T)\in \MM \mbox{ with }x\mbox{ contractible} \}\\
&= - \inf \Bigl\{ \frac{1}{T} \int_0^T L(\gamma(t),\gamma'(t)) \, dt \, \Big| \, \gamma\in C^{\infty}(\R/T\Z,M) \mbox{ contractible, } T>0 \Bigr\}, \\
e_0(L)&:=\max_{x\in M} E(x,0) = \max \set{E(x,v)}{(x,v) \in \Crit E}.
\end{split} \end{equation*}
The number $c_0(L)$ is known as the {\em strict Ma\~n\'e critical value}, while $c_u(L)$ is the {\em lowest Ma\~{n}\'e critical value} (see \cite{man97}). When the fundamental group of $M$ is rich, there are other Ma\~n\'e critical values, which are associated to the different coverings to $M$, but the above ones are those which are more relevant for the question of existence of periodic orbits.
It is easy to see that
\[
\min E\leq e_0(L)\leq c_u(L)\leq c_0(L).
\]
When $L$ has the form \eqref{elmag}, $\min E$ is the minimum of the scalar potential $V$, while $e_0(L)$ is its maximum. When the magnetic potential $\theta$ vanishes, the identities $e_0(L)=c_u(L)=c_0(L)$ hold, but in general $e_0(L)$ is strictly lower than the other two numbers. The values $c_u(L)$ and $c_0(L)$ coincide when the fundamental group of $M$ is Abelian and, more generally, when it is ameanable (see \cite{fm07}).

The lowest Ma\~{n}\'e critical value $c_u(L)$ plays a decisive role in the geometry of the action functional $\SSS_{\kappa}$, as the next result shows:

\begin{Lemma}\label{Lem:5}
If $\kappa\geq c_u(L)$, then $\SSS_{\kappa}$ is bounded from below on every connected component of $\MM$. If $\kappa<c_u(L)$, then the functional $\SSS_{\kappa}$ is unbounded from below on each connected component of $\MM$.
\end{Lemma}

\begin{proof}
Choose $\gamma:\R/T\Z\rightarrow M$ in some connected component of the free loop space and let $\tilde\gamma:[0,T]\pf\widetilde M$ be the its lift to the universal covering $\pi:\widetilde M \rightarrow M$. We lift the metric of $M$ to $\widetilde{M}$ and notice that the fact of having fixed the connected component of the free loop space implies that  the quantity $\mathrm{dist}\big(\tilde\gamma(T),\tilde\gamma(0)\big)$ is uniformly bounded. Therefore, there exists a path $\tilde{\alpha}:[0,1]\rightarrow \widetilde M$ which joins $\tilde{\gamma}(T)$ to $\tilde{\gamma}(0)$ and has uniformly bounded action
\[
\widetilde{\A}_{\kappa} (\tilde{\alpha}) := \int_0^1 \Bigl( \widetilde L\bigl(\tilde{\alpha}(t), \tilde{\alpha}'(t) \bigr) + \kappa \Bigr) \, dt \leq C,
\]
where $\widetilde{L}$ denotes the Lagrangian on $T\widetilde M$ which is obtained by lifting $L$. If $\alpha := \pi\circ \tilde{\alpha}$, the juxtaposition $\gamma\# \alpha$ is a contractible loop in $M$. Since $\kappa\geq c_u(L)$, we have
\[
0 \leq \SSS_{\kappa} (\gamma \# \alpha) = \SSS_{\kappa} (\gamma) + \SSS_{\kappa} (\alpha) = \SSS_{\kappa} (\gamma) + \widetilde{\A}_{\kappa} (\tilde{\alpha})
\leq \SSS_{\kappa} (\gamma)+C,
\]
from which $\SSS_{\kappa}(\gamma)\geq -C$.

When $\kappa<c_u(L)$, the functional $\SSS_{\kappa}$ is unbounded from below on each connected component of $\MM$. In fact, if $\alpha$ is a contractible closed curve with $\SSS_{\kappa}(\alpha)<0$, we can modify any closed curve $\gamma$ within its free homotopy class and make it have arbitrarily low action $\SSS_{\kappa}$: Join $\gamma(0)$ to $\alpha(0)$ by some path, wind around $\alpha$ several times, come back to $\gamma(0)$ by the inverse path, and finally go once around $\gamma$.
\end{proof}

The strict Ma\~n\'e critical value $c_0(L)$ is not directly related to the geometry of $\SSS_{\kappa}$, but has the following important characterization (see \cite{fat97} and \cite{cipp98}):
\begin{equation}
\label{cara}
c_0(L) = \inf \left\{ \max_{x\in M} H(x,\alpha(x)) \, \Big| \, \alpha \mbox{ smooth closed one-form on } M\right\},
\end{equation}
where $H:T^*M \rightarrow \R$ is the Hamiltonian associated to the Lagrangian $L$ via Legendre duality:
\[
H(x,p) := \max_{v\in T_x M} \bigl( p[v] - L(x,v) \bigr).
\]
Then $H$ is a Tonelli Hamiltonian, meaning that it is fiberwise superlinear and uniformly convex (see the beginning of Section \ref{sec3}). 
Let $X_H$ be the induced Hamiltonian vector field on $T^*M$, which is defined by the identity
\[
\omega(X_H(z),\zeta) = - dH(z)[\zeta], \quad \forall z\in T^*M, \; \zeta\in T_z T^*M,
\]
where $\omega=dp\wedge dx$ is the standard symplectic form on $T^*M$. The flow of $X_H$ preserves each level $H^{-1}(\kappa)$, where it is conjugated to the Euler-Lagrange flow of $L$ on $E^{-1}(\kappa)$ by the Legendre transform
\[
TM \rightarrow T^*M, \quad (x,v) \mapsto \bigl( x , d_v L(x,v) \bigr).
\]
Assume that $\kappa$ is a regular value of $H$, so that $\Sigma:= H^{-1}(\kappa)$ is a hypersurface. Up to a time reparametrization, the dynamics on $\Sigma$ is determined only by the geometry of $\Sigma$ and not by the Hamiltonian of which $\Sigma$ is an energy level: In fact the nowhere vanishing  vector field $X_H|_{\Sigma}$ belongs to the one-dimensional distribution 
\[
\mathcal{L}_{\Sigma}:=\ker \omega|_{\Sigma},
\]
whose integral curves are hence the orbits of $X_H|_{\Sigma}$.

The characterization (\ref{cara}) has a dynamical and a geometric consequence. We begin with the dynamical one: 

\begin{thm}
\label{finsler}
If $\kappa>c_0(L)$, then the Euler-Lagrange flow on $E^{-1}(\kappa)$ is conjugated up to a time-reparametrization to the geodesic flow which is induced by a Finsler metric on $M$.
\end{thm}

\begin{proof}
Since $\kappa>c_0(L)$, there is a smooth closed one-form $\alpha$ whose image is contained in the sublevel $\{H<\kappa\}$. Since $\alpha$ is closed, the diffeomorphism of $T^*M$ defined by $(x,p) \mapsto (x,p+\alpha(x))$ is symplectic and conjugates the Hamiltonian flow of $H$ to that of $K(x,p):= H(x,p+\alpha(x))$. The energy level $K^{-1}(\kappa)$ is now the boundary
of a fiberwise uniformly convex bounded open set which contains the zero section of $T^*M$. Therefore, there exists a fiberwise convex and 2-homogeneous function $F:T^*M \rightarrow [0,+\infty)$ such that $F^{-1}(1)=K^{-1}(\kappa)$. Thus, the Hamiltonian flow of $F$ on $F^{-1}(1)=K^{-1}(\kappa)$ is related to that of $K$ - hence to that of $H$ on $H^{-1}(\kappa)$ - by a time reparametrization. But the Legendre dual of the fiberwise convex and 2-homogeneous Hamiltonian $F$ is the square of a Finsler structure on $M$. We conclude that the orbits of the Euler-Lagrange flow of $L$ of energy $\kappa$ are reparametrized Finsler geodesics.
\end{proof}

In order to derive the geometric consequence of (\ref{cara}), we need to recall some notions from symplectic topology.

\medskip

\paragraph{\bf Contact type and stable energy hypersurfaces} 
The energy level $\Sigma$ is said to be of {\em contact type} if there is a one-form $\eta$ on $\Sigma$ which is a primitive of $\omega|_{\Sigma}$ and is such that $\eta$ does not vanish on $\mathcal{L}_{\Sigma}$. Equivalently, there is a smooth vector field $Y$ in a neighborhood of $\Sigma$ which is transverse to $\Sigma$ and such that $L_Y \omega = \omega$ (the vector field $Y$ and the one-form $\eta$ are related by the identity $\imath_Y \omega|_{\Sigma} = \eta$). The energy level $\Sigma$ is said to be of {\em restricted contact type} if the one-form $\eta$ extends to a primitive of $\omega$ on the whole $T^*M$. 

If the surface $\Sigma\subset T^*M$ bounds an open fiberwise convex set which contains the zero-section (or more generally an open set which is starshaped with respect to the zero section), then it is of restricted contact type: As $\eta$ one can take the Liouville form $p\, dq$. Therefore, arguing as in the first part of the proof of Theorem \ref{finsler}, we obtain the following geometric consequence of the characterization (\ref{cara}):

\begin{thm}
\label{contact}
If $\kappa>c_0(L)$, then $H^{-1}(\kappa)$ is of restricted contact type.
\end{thm}

 If $c_u(L) \leq \kappa \leq c_0(L)$ and $M$ is not the 2-torus, $H^{-1}(\kappa)$ is not of contact type (see \cite[Proposition B.1]{con06}), and it is conjectured that the same is true for $e_0(L)<\kappa<c_u(L)$. If $\min E < \kappa < e_0(L)$, $H^{-1}(\kappa)$ might or might not be of contact type: For instance, if the one-form $\theta(x)[v]:=d_vL(x,0)[v]$ is closed, then every regular energy level is of contact type (see \cite[Proposition C.2]{con06}, in this case $e_0(L)=c_u(L)=c_0(L)$).

The contact condition has the following dynamical consequence: If $\Sigma$ is a contact type compact hypersurface in a symplectic manifold $(W,\omega)$ (in our case, $W=T^*M$), then there is a diffeomorphism
\[
(-\epsilon, \epsilon) \times \Sigma \rightarrow W, \qquad (r,x) \mapsto \psi_r(x),
\]
onto an open neighborhood of $\Sigma$ such that $\psi_0$ is the identity on $\Sigma$ and 
\[
\psi_r : \Sigma \rightarrow \Sigma_r := \psi_r(\Sigma)
\]
induces an isomorphism between the line bundles $\mathcal{L}_{\Sigma}$ and $\mathcal{L}_{\Sigma_r}$ (the hypersurface $\Sigma_r$ is the image of $\Sigma$ by the flow at time $r$ of the vector field $Y$ given by the contact condition, see \cite[page 122]{hz94}). Therefore, if the hypersurfaces $\Sigma_r$ are level sets of a Hamiltonian $K$, the dynamics of $X_K$ on $\Sigma_r$ is conjugate to the one on $\Sigma_0=\Sigma$ up to a time reparametrization.

Hypersurfaces with the above propery are called {\em stable} (see \cite[page 122]{hz94}). The stability condition is weaker than the contact condition, as the following characterization, which is due to K.\ Cieliebak and K.\ Mohnke \cite[Lemma 2.3]{cm05}, shows:

\begin{Prop}   
Let $\Sigma$ be a compact hypersurface in the symplectic manifold $(W,\omega)$. Then the following facts are equivalent:
\begin{enumerate}
\item $\Sigma$ is stable;
\item there is a vector field $Y$ on a neighborhood of $\Sigma$ which is transverse to $\Sigma$ and satisfies $\mathcal{L}_{\Sigma} \subset \ker ( L_Y \omega|_{\Sigma})$;
\item there is a one-form $\eta$ on $\Sigma$ such that $\mathcal{L}_{\Sigma} \subset \ker d\eta$ and $\eta$ does not vanish on $\mathcal{L}_{\Sigma}$.
\end{enumerate}
\end{Prop}

\begin{proof}
(i) $\Rightarrow$ (ii). By stability, a neighborhood of $\Sigma$ can be identified with $(-\epsilon,\epsilon)\times \Sigma$ in such a way that $\mathcal{L}_{\{r\} 
\times \Sigma}$ does not depend on $r$. Set $Y:= \partial/\partial r$ and denote by $\phi_t(r,x)=(r+t,x)$ its flow. Then $\ker ( \phi_t^* \omega|_{\{0\} \times \Sigma} )$ does not depend on $t$ and differentiating in $t$ at $t=0$ we get
\[
\mathcal{L}_{\Sigma} = \ker \omega|_{\Sigma} \subset \ker ( L_Y \omega|_{\Sigma}).
\]
(ii) $\Rightarrow$ (iii). If we set $\eta:= \imath_Y \omega|_{\Sigma}$, by Cartan's identity we have
\[
d\eta = d \imath_Y \omega|_{\Sigma} = (L_Y \omega - \imath_Y d \omega)|_{\Sigma} =  L_Y \omega|_{\Sigma},
\]
so $\mathcal{L}_{\Sigma} \subset \ker(L_Y \omega|_{\Sigma}) = \ker d\eta$. If $\xi\neq 0$ is a vector in $\mathcal{L}_{\Sigma}$, then
\[
\eta(\xi) = \omega(Y,\xi) \neq 0,
\]
because $Y\notin T\Sigma$.

\noindent (iii) $\Rightarrow$ (i). Consider the closed two-form on $(-\epsilon,\epsilon) \times \Sigma$
\[
\tilde{\omega} = \omega|_{\Sigma} + d(r\eta) = \omega|_{\Sigma} + rd\eta + dr\wedge \eta.
\]
If $\epsilon$ is small enough, the form $\omega|_{\Sigma} + rd\eta$ is non-degenerate on $\ker \eta$ for every $r\in (-\epsilon,\epsilon)$, from which we deduce that $\tilde{\omega}$ is a symplectic form. Since $\tilde{\omega}|_{\{0\} \times \Sigma}$ coincides with $\omega|_{\Sigma}$, by the coisotropic neighborhood theorem (see \cite{got82}, or \cite[Exercise 3.36]{ms98} for the particular case of a hypersurface), a neighborhood of $\Sigma$ in $W$ is symplectomorphic to $((-\epsilon,\epsilon) \times \Sigma,\tilde{\omega})$, up to the choice of a smaller $\epsilon$. Since for $\xi\in \mathcal{L}_{\Sigma}(x)$ and $\zeta \in T_{(r,x)} (\{r\} \times \Sigma) = (0) \times T_x \Sigma$ there holds
\[
\tilde{\omega}(\xi,\zeta) = \omega(\xi,\zeta) + r d\eta (\xi,\zeta) = 0,
\]
we deduce that $\ker (\tilde{\omega}|_{\{r\} \times \Sigma}) = \mathcal{L}_{\Sigma}$ does not depend on $r$. Therefore, $\{0\} \times \Sigma$ is stable in $((-\epsilon,\epsilon) \times \Sigma,\tilde{\omega})$ and hence $\Sigma$ is stable in $(W,\omega)$.
\end{proof}

\begin{Rmk}
L.~Macarini and G.~P.~Paternain have constructed examples of Tonelli Lagrangians on the tangent bundle of $\mathbb{T}^n$ such that $H^{-1}(\kappa)$ is stable for $\kappa=c_u(L)=c_0(L)$, see \cite{mp10}.
\end{Rmk}

\section{Palais-Smale sequences}
\label{pss}

Palais-Smale sequences $(x_h,T_h)$ with $T_h\rightarrow 0$ are a possible source of non-compactness, but the next result shows that they occur only at level zero.

\begin{Lemma}
\label{PST0}
Let $(x_h,T_h)$ be a $(\PS)_c$ sequence for $\SSS_{\kappa}$ with $T_h\to0$. Then $c=0$.
\end{Lemma}

\begin{proof}
Set $\gamma_h(t):= x_h(t/T_h)$.
Since $(x_h,T_h)$ is a $(\PS)$ sequence, the identity (\ref{eq:diff A w.r.t. T}) shows that the sequence
\[
\alpha_h := \frac{1}{T_h} \int_0^{T_h} \bigl( E(\gamma_h,\gamma_h') - \kappa \bigr) \, dt = - \frac{\partial \SSS_{\kappa}}{\partial T} (x_h,T_h)
\]
is infinitesimal and, in particular, bounded. Since $L(x,v)$ is electromagnetic for $|v|$ large, $E(x,v)$ has the form (\ref{ene}) for $|v|$ large and hence satisfies the estimate
\[
E(x,v) \geq E_0 |v|^2 - E_1, \qquad \forall (x,v)\in TM,
\]
for suitable numbers $E_0>0$ and $E_1\in \R$. It follows that
\[
\alpha_h \geq \frac{1}{T_h} \int_0^{T_h} \bigl( E_0 |\gamma_h'|^2 - E_1 - \kappa \bigr)\, dt = \frac{E_0}{T_h} \int_0^{T_h} |\gamma_h'|^2\,dt - (E_1 + \kappa),
\]
and hence
\[
\int_0^{T_h}  |\gamma_h'|^2\,dt \leq \frac{T_h}{E_0} ( \alpha_h + E_1 + \kappa) = O(T_h)
\]
for $h\rightarrow \infty$. From the lower bound (\ref{boundsonL}) and from the analogous upper bound
\[
L(x,v) \leq L_2 |v|^2 + L_3, \qquad \forall (x,v)\in TM,
\]
for suitable positive number $L_2,L_3$, we find
\[
\SSS_{\kappa}(x_h,T_h) = \int_0^{T_h} \bigl( L(\gamma_h,\gamma_h') + \kappa \bigr)\, dt \geq L_0 \int_0^{T_h} |\gamma_h'|^2\, dt + T_h (\kappa - L_1) = O(T_h)
\]
and
\[
\SSS_{\kappa}(x_h,T_h) = \int_0^{T_h} \bigl( L(\gamma_h,\gamma_h') + \kappa \bigr)\, dt \leq L_2 \int_0^{T_h} |\gamma_h'|^2\, dt + T_h (L_3 + \kappa) = O(T_h)
\]
for $h\rightarrow \infty$. Since $T_h\rightarrow 0$, we conclude that $\SSS_{\kappa}(x_h,T_h)$ is infinitesimal and hence $c=0$.
\end{proof}

\begin{Rmk}
By choosing a suitable metric on the Hilbert manifold $\mathcal{M}$, one can also obtain that for every Palais-Smale sequences $(x_h,T_h)$ with $T_h\rightarrow 0$ the sequence $(x_h)$ converges to an equilibrium solution of the Euler-Lagrange equations which has energy $\kappa$ (see \cite[Proposition 3.8]{con06}). In particular, such Palais-Smale sequences can exist only when $\kappa$ is a critical value of $E$.
\end{Rmk}

The next result says that Palais-Smale sequences $(x_h,T_h)$ with $(T_h)$ bounded and bounded away from zero are always compact.

\begin{Lemma}\label{Lem;2}
Let $(x_h,T_h)$ be a $(\PS)_c$ sequence for $\SSS_{\kappa}$ with $0<T_*\leq T_h\leq T^*<\infty$. Then $(x_h,T_h)$ is compact in $\MM$.
\end{Lemma}

\begin{proof}
Up to a subsequence, we may assume that $(T_h)$ converges to some $T\in [T_*,T^*]$.
By (\ref{boundsonL}) we have
\begin{equation}
\label{isbd}
\begin{split}
c+o(1) = \SSS_{\kappa}(x_h,T_h) = T_h \int_0^1 \Bigr( L\bigl(x_h,x_h'/T_h \bigr)+\kappa \Bigr)\, ds \\ \geq T_h \int_0^1\Bigr( L_0 \frac{|x_h'|^2}{T_h^2}-(L_1-\kappa)\Bigr)\, ds \geq \frac{L_0}{T^*}\|x_h'\|_2^2-T^* |L_1-\kappa|,
\end{split} \end{equation}
where $\|\cdot\|_2$ denotes the $L^2$ norm with respect to the fixed Riemannian metric on $M$. Therefore, $\|x_h'\|_2$ is uniformly bounded and $(x_h)$ is 1/2-equi-H\"older-continuous:
\[
\mathrm{dist}\big(x_h(s'),x_h(s)\big)\leq\int_s^{s'}|x_h'(r)|\, dr \leq |s'-s|^{1/2} \|x_h'\|_2.
\]
By the Ascoli-Arzel\`a theorem, up to a subsequence $(x_h)$ converges uniformly to some $x\in C(\T,M)$.  In particular, $(x_h)$ eventually belongs to the image of the parameterization $\varphi_*$ induced by a smooth map
\[
\varphi: \T\times B_r \rightarrow M.
\]
See \eqref{eq:chart of the loop space} and recall that the image of this parameterization is $C^0$-open.
Then $x_h=\varphi_*(\xi_h)$, where $\xi_h$ belongs to $W^{1,2}(\T,B_r)$ and is a $(\PS)$ sequence for the functional
$$
\widetilde\A(\xi,T)=T\int_{\T}\widetilde L\bigr(s,\xi,\xi'/T \bigr)\,ds,
$$
with respect to the standard Hilbert product on $W^{1,2}(\T,\R^n)$,
where the Lagrangian $\widetilde{L}\in C^{\infty}(\T\times B_r \times \R^n)$ is obtained by pulling back $L+\kappa$ by $\varphi$. Moreover, $(\xi_h)$ converges uniformly and, since $\|\xi_h'\|_2$ is bounded, weakly in $W^{1,2}$ to some $\xi$ in
$W^{1,2}(\T,B_r)$. We must prove that this convergence is actually strong in $W^{1,2}$.

Since $\widetilde{L}(s,x,v)$ is electromagnetic for $|v|$ large, we have the bounds
\begin{equation}
\label{bddL}
\bigl| d_x \widetilde L (s,x,v) \bigr| \leq C(1+|v|^2),\quad\bigl| d_v \widetilde L (s,x,v)\bigr| \leq C(1+|v|),
\end{equation}
for a suitable constant $C$. Since $(\xi_h,T_h)$ is a $(\PS)$ sequence with $(\xi_h)$  bounded in $W^{1,2}$, we have by (\ref{diff1})
\begin{equation*}
\begin{split}
o(1) &= d\widetilde\A(\xi_h,T_h)[(\xi_h-\xi,0)] \\ &= T_h \int_{\T} d_x \widetilde L \bigl(s,\xi_h,\xi_h'/T_h\bigr) [\xi_h-\xi]\, ds + T_h \int_{\T} d_v \widetilde{L} \bigl(s,\xi_h,\xi_h'/T_h\bigr) \bigl[ (\xi_h'-\xi')/T_h \bigr] \, ds.
\end{split} \end{equation*}
By the first bound in (\ref{bddL}) and the uniform convergence $\xi_h\rightarrow \xi$, the first integral is infinitesimal. Together with the fact that $(T_h)$ is bounded away from zero, we obtain
\begin{equation}
\label{uno}
\int_{\T} d_v \widetilde{L} \bigl(s,\xi_h,\xi_h'/T_h\bigr) \bigl[ (\xi_h'-\xi')/T_h \bigr] \, ds=o(1).
\end{equation}
From the fiberwise uniform convexity of $\widetilde{L}$, we have the bound
\[
d_{vv} \widetilde{L}  (s,x,v) [u, u] \geq \delta |u|^2, \quad \forall (s,x,v)\in \T\times B_r \times \R^n, \; u\in \R^n,
\]
for a suitable positive number $\delta$. It follows that
\begin{equation*}
\begin{split}
d_v \widetilde L \left(s,\xi_h,\frac{\xi_h'}{T_h}\right) & \left[ \frac{\xi_h'-\xi'}{T_h} \right] - d_v\widetilde L  \left(s,\xi_h,\frac{\xi'}{T_h} \right)\left[\frac{\xi_h'-\xi'}{T_h} \right] \\
&= \int_0^1 d_{vv} \widetilde L \left(s,\xi_h,\frac{\xi'}{T_h}+\sigma \frac{\xi_h'-\xi'}{T_h} \right) \left[ \frac{\xi_h'-\xi'}{T_h}, \frac{\xi_h'-\xi'}{T_h} \right] \, d\sigma \geq \frac{\delta}{T_h^2} |\xi_h'- \xi'|^2.
\end{split} \end{equation*}
By integrating this inequality over $s\in \T$ and by (\ref{uno}), we obtain
\[
o(1)- \int_{\T} d_v \widetilde L \bigl(s,\xi_h,\xi'/T_h)\bigl[ (\xi_h'-\xi')/T_h \bigr] \, ds \geq \frac{\delta}{T_h^2} \|\xi_h' - \xi'\|_2^2.
\]
By the second bound in (\ref{bddL}), the sequence
\[
d_v \widetilde L \bigl(s,\xi_h,\xi'/T_h\bigr)
\]
converges strongly in $L^2$. By the weak $L^2$ convergence to $0$ of $(\xi_h'-\xi)$, we deduce that the integral on the left-hand side of the above inequality is infinitesimal. We conclude that $(\xi_h)$ converges to $\xi$ strongly in $W^{1,2}$.
\end{proof}

In general, $\SSS_{\kappa}$ might have Palais-Smale sequences $(x_h,T_h)$ with $(T_h)$ unbounded. However, this does not occur when $\kappa$ is larger than the lowest Ma\~{n}\'e critical value $c_u(L)$:

\begin{Lemma}\label{Lem:6}
If $\kappa>c_u(L)$, then any $(\PS)$ sequence $(x_h,T_h)$ in a given connected component of $\MM$ with $T_h\geq T_*>0$ is compact.
\end{Lemma}

\begin{proof}
By Lemma \ref{Lem;2}, it is enough to show that $(T_h)$ is bounded from above. Since
\[
\SSS_{\kappa} (x,T)=\SSS_{c_u(L)} (x,T) + \big(\kappa -c_u(L)\big)T,
\]
the period
\[
T_h = \frac{1}{\kappa -c_u(L)} \bigr(\SSS_{\kappa} (x_h,T_h) - \SSS_{c_u(L)} (x_h,T_h)\bigr)
\]
is bounded from above, because $\SSS_{\kappa}$ is bounded on the $(\PS)$ sequence $(x_h,T_h)$ and $\SSS_{c_u(L)}(x_h,T_h)$ is bounded from below by Lemma \ref{Lem:5}.
\end{proof}

\begin{Rmk}
By choosing a suitable metric on $\mathcal{M}$, it is possible to characterize $c_u(L)$ as the infimum of all $\kappa_0$'s such that $\SSS_{\kappa}$ satisfies the Palais-Smale condition for every $\kappa\in [\kappa_0,+\infty)$. See \cite{cipp00}.
\end{Rmk} 

\section{Periodic orbits with high energy}
\label{pohe}

The following result shows that the energy levels above $c_u(L)$ have always periodic orbits and proves statements (i) and (ii) of the theorem in the Introduction.

\begin{Thm}
\label{highthm}
Assume that $\kappa>c_u(L)$. Then the following facts hold.
\begin{enumerate}
\item If $M$ is not simply connected, then the energy level $E^{-1}(\kappa)$ has a $\kappa$-action-minimizing periodic orbit in each non-trivial homotopy class of the free loop space.
\item If $M$ is simply connected, then the energy level $E^{-1}(\kappa)$ has a periodic orbit with positive  action $\SSS_{\kappa}$.
\end{enumerate}
\end{Thm}

\begin{proof}
(i) Assume that $M$ is not simply connected. Let $\alpha \in[\T,M]$ be a non-trivial homotopy class and let
$\MM_{\alpha}$ be the connected component of $\MM^\mathrm{noncontr}$ corresponding to $\alpha$. By Lemma \ref{Lem:5}, the functional $\SSS_{\kappa}$ is bounded from below on $\MM_{\alpha}$. By Lemma \ref{Lem:3} (i), the sublevels
\[
\{ (x,T)\in \MM_{\alpha} \, | \, \SSS_{\kappa}(x,T)\leq c \}
\]
are complete. Let $(x_h,T_h)\subset \MM_{\alpha}$ be a (PS) sequence for $\SSS_{\kappa}$. By Lemma \ref{Lem:3} (i), $(T_h)$ is bounded away from zero, so Lemma \ref{Lem:6} implies that $\SSS_{\kappa}$ satisfies the (PS) condition on $\MM_{\alpha}$. By Remark \ref{mini}, we conclude that $\SSS_{\kappa}$ has a minimizer on $\MM_{\alpha}$, as we wished to prove.

\medskip

(ii) Assume that $M$ is simply connected, so that $\MM=\MM^{\mathrm{contr}}$. In this case, $\SSS_{\kappa}$ is strictly positive everywhere, because $\kappa>c_u(L)$, but the infimum of $\SSS_{\kappa}$ is zero, as one readily checks by looking at sequences of the form $(x_0,T_h)$, with $x_0$ a constant loop and $T_h\rightarrow 0$. So the infimum is not achieved. We will find the periodic orbit by considering the same minimax class which Lusternik and Fet \cite{lf51} considered in their proof of the existence of a closed geodesic on a simply connected compact manifold.

Since the closed manifold $M$ is simply connected, there exists $l\geq 2$ such that $\pi_l(M)\ne 0$ (a manifold all of whose homotopy groups vanish is contractible, but closed manifolds are never contractible, for instance because their $n$-dimensional homology group with $\Z_2$ coefficients does not vanish). We fix a non-zero homotopy class $\mathcal{G} \in\pi_l(M)$. Thanks to the isomorphism
$\pi_{l-1}(C(\T,M))\cong\pi_l(M)$, we have an induced non-zero homotopy class
\[
\mathcal{H} \in [S^{l-1},C(\T,M)] \cong [S^{l-1},\MM],
\]
and we consider the minimax value
\[
c=\inf_{\substack{h:S^{l-1}\to \MM\\ h\in \mathcal{H}}} \max_{\xi\in S^{l-1}} \SSS_{\kappa} (h(\xi)).
\]
Let us show that $c>0$. Since $\H$ is non-trivial, there exists a positive number $a$ such that for every map $h=(x,T): S^{l-1} \rightarrow \MM$ belonging to the class $\H$ there holds
\[
\max_{\xi \in S^{l-1}} \ell(x(\xi)) \geq a,
\]
where $\ell(x(\xi))$ denotes the length of the loop $x(\xi)$ (see \cite[Theorem 2.1.8]{kli78}). If $(x,T)$ is an element of $\MM$ with $\ell(x)\geq a$, then (\ref{boundsonL}) implies
\begin{equation*}
\begin{split}
\SSS_{\kappa}(x,T) &= T\int_{\T} \Bigl(L\bigl(x,x'/T\bigr)+\kappa \Bigr) \,ds
\geq T \int_{\T} \Bigl( L_0\frac{|x'|^2}{T^2} -L_1+\kappa \Bigr)\, ds \\
& \geq \frac{L_0}{T}\ell(x)^2-T(L_1-\kappa)
\geq \frac{L_0}{T}a^2-T(L_1-\kappa).
\end{split} \end{equation*}
Since $a>0$, the above chain of inequalities implies that there exists $T_0>0$ such that for every $(x,T)\in \MM$ with $\ell(x)\geq a$ and $\SSS_{\kappa}(x,T)\leq c+1$, the period $T$ is at least $T_0$. Now let $h\in \H$ be such that the maximum of $\SSS_{\kappa}$ on $h(S^{l-1})$ is less than $c+1$. By the above considerations, there exists $(x,T)$ in $h(S^{l-1})$ with $T\geq T_0$, whence
\[
\SSS_{\kappa} (x,T)=\SSS_{c_u(L)} (x,T)+ \bigl(\kappa-c_u(L)\bigr)T \geq \bigl(\kappa-c_u(L)\bigr) T_0.
\]
This shows that the minimax value $c$ is strictly positive.

Theorem \ref{thm:finite c induces PS}, together with Remark \ref{trunc} and Lemma \ref{Lem:4}, implies the existence of a $(\PS)_c$ sequence $(x_h,T_h)$. Lemma \ref{PST0} guarantees that $(T_h)$ is bounded away from zero, so by Lemma \ref{Lem:6} the sequence $(x_h,T_h)$ has a limiting point in $\MM$, which gives us the required periodic orbit.
\end{proof}

\begin{Rmk}
If $M$ is not simply connected and $\kappa>c_u(L)$, the energy level $E^{-1}(\kappa)$ might have no contractible periodic orbits. Consider for instance the Lagrangian $L(x,v)=|v|^2/2$ on the torus $\T^n$, equipped with the flat metric. The corresponding Euler-Lagrange flow is the geodesic flow on $T\T^n$, $c_u(L)=0$, and all the non-constant closed geodesics on the flat torus are non-contractible.
\end{Rmk}

\begin{Rmk}
\label{veryhigh}
If $\kappa>c_0(L)$, the existence of a periodic orbit on $E^{-1}(\kappa)$ also follows from Theorem \ref{finsler}, because every Finsler metric on a closed manifold has a closed geodesic.  
\end{Rmk}

\begin{Rmk}
Theorem \ref{finsler} implies that the multiplicity results for closed Finsler geodesics hold also for Hamiltonian orbits at energy levels $\kappa>c_0(L)$. Actually, most of these results remain true for $\kappa>c_u(L)$.
In fact, as the proof of Theorem \ref{highthm} suggests, the (PS) condition and the topology of the sublevels of the functional $\SSS_{\kappa}$ are analogous to the corresponding properties of the Finsler geodesic energy functional (with the notable exception of the zero level). 
\end{Rmk}

\section{Topology of the free period action functional for low energies}
\label{tfpafle}

As we have seen, when $\kappa<c_u(L)$ the functional $\SSS_{\kappa}$ is unbounded from below on each connected component of $\MM$. The aim of this section is to show that, nevertheless, some sublevels of $\SSS_{\kappa}$ have a sufficiently rich topology.

We start by proving a simple lemma about the integral of a one form. The integral of a given one-form on a curve $x$ is clearly bounded by a constant times the length of $x$. When the support of the curve is contained in a ball of $M$, one may also take the square of the length in this bound, which is a better estimate for short curves. More precisely, we have the following:

\begin{Lemma}\label{Lemma:isoper}
Let $\theta$ be a smooth one-form on $M$ and let $U\subset M$ be an open set whose closure is diffeomorphic to a closed ball in $\R^n$. Then there exists a number $\Theta>0$ such that
\[
\bigg| \int_\T x^*(\theta) \bigg| \leq \Theta \cdot\ell(x)^2,
\]
for every closed curve $x:\T\rightarrow U$.
\end{Lemma}

\begin{proof}
Up to the change of the constant $\Theta$,
we may assume that $U=B_r$ is the ball of radius $r$ around the origin in $\R^n$, equipped with the Euclidean metric. Given the closed curve $x:\T\rightarrow B_r$, we consider the map
\[
X:[0,1]\times \T \rightarrow B_r, \quad X(s,t) = x(0) + s \bigl( x(t) - x(0) \bigr).
\]
Then $X(1,t)=x(t)$ and $X(0,t)=x(0)$, hence by Stokes theorem
\begin{equation*}
\begin{split}
\left| \int_{\T} x^*(\theta) \right| &= \bigg| \int_{[0,1] \times \T} X^* (d\theta) \bigg| = \left| \int_0^1 ds \int_{\T} d\theta\bigl( X(s,t) \bigr) \Bigl[ \frac{\p X}{\p s} , \frac{\p X}{\p t} \Bigr] \, dt \right| \\ &\leq \|d\theta\|_{\infty} \int_0^1 ds \int_{\T} \left| \frac{\p X}{\p s} \right| \, \left| \frac{\p X}{\p t} \right| \, dt
= \|d\theta\|_{\infty} \int_0^1 ds \int_{\T} \bigl| x(t) - x(0) \bigr| s \bigl| x'(t) \bigr|\, dt \\ &\leq \frac{1}{2} \|d\theta\|_{\infty}\, \ell(x) \int_0^1 ds \int_{\T} s \bigl| x'(t)\bigr|\, dt = \frac{1}{4} \|d\theta\|_{\infty}\, \ell(x)^2,
\end{split} \end{equation*}
as claimed.
\end{proof}

\paragraph{\bf The energy range $\mathbf{(e_0(L),c_u(L))}$} If $\kappa<c_u(L)$, there are contractible closed curves with negative action $\SSS_{\kappa}$. Since the space of contractible loops is connected, we can consider the following class of continuous paths in $\MM$:
\begin{equation}
\label{curve}
\mathcal{Z}_0 := \bigl\{ (x,T):[0,1] \rightarrow \MM\, \big| \, x(0) \mbox{ is a constant loop and } \SSS_{\kappa}(x(1),T(1)) <0 \bigr\}.
\end{equation}
Notice that if $x_0$ is a constant loop and $T>0$, then
\begin{equation}
\label{constants}
\SSS_{\kappa}(x_0,T) = T \bigl( L(x_0,0) + \kappa \bigr) = T \bigl( \kappa - E(x_0,0) \bigr).
\end{equation}
When $\kappa>e_0(L)=\max_{x\in M} E(x,0)$, the above quantity is strictly positive (and tends to zero for $T\rightarrow 0$). The next lemma shows that when $e_0(L) < \kappa < c_u(L)$, $\SSS_{\kappa}$ has a sort of mountain pass geometry:

\begin{Lemma}
\label{Lem:mountain pass}
Assume that $e_0(L) <\kappa < c_u(L)$. Then there exists $a>0$ such that for every $z\in \mathcal{Z}_0$ there holds
\[
\max_{\sigma\in[0,1]} \SSS_{\kappa} \bigl(z(\sigma)\bigr)\geq a.
\]
\end{Lemma}

\begin{proof}
Consider the smooth one-form on $M$,
\[
\theta(x) [v] := d_v L(x,0)[v].
\]
By taking a Taylor expansion and by using the bound (\ref{bd2L}), we get the estimate
\begin{equation}
\label{bbb}
L(x,v) = L(x,0)+d_vL(x,0)[v]+\frac{1}{2}d_{vv}L(x,s v)[v,v]
\geq -E(x,0)+\theta(x)[v]+ L_0 |v|^2,
\end{equation}
where $s\in [0,1]$. Let $\{U_1,\dots,U_N\}$ be a finite covering of $M$ consisting of open sets whose closures are diffeomorphic to closed Euclidean balls, and let $\Theta >0$ be such that the conclusion of Lemma
\ref{Lemma:isoper} holds for the one-form $\theta$, for each of the open sets $U_j$'s. Let $r_0$ be a Lebesgue number for this covering, meaning that every ball of radius $r_0$ is contained in one of the $U_j$'s.

We claim that if $\SSS_{\kappa}(x,T)<0$ then
\begin{equation}
\label{claim}
\ell(x)>\min\bigg\{r_0,\frac{\sqrt{L_0(\kappa-e_0(L))}}{\Theta}\bigg\}=:r_1.
\end{equation}
In fact, assuming that $\ell(x)\leq r_0$, we have that $x(\T)$ is contained in some $U_j$, for $1\leq j \leq N$. Set as usual $\gamma(t)=x(t/T)$.
By Lemma \ref{Lemma:isoper} and by (\ref{bbb}), we obtain the chain of inequalities
\begin{equation}
\label{stma}
\begin{split}
0 &>\SSS_{\kappa}(x,T) =\SSS_{\kappa}(\gamma) = \int_0^T \bigr(L(\gamma,\gamma')+\kappa \bigr)\,dt
\\ &\geq \int_0^T\bigr( -E(\gamma,0) + \theta(\gamma)[\gamma'] + L_0 |\gamma'|^2 + \kappa \bigr)\, dt
\\ &= \int_0^T \bigl( \kappa - E(\gamma,0)\bigr)\, dt + \int_{\R/T\Z} \gamma^*(\theta) + L_0 \int_0^T |\gamma'|^2\, dt \\
&\geq \bigl(\kappa-e_0(L)\bigr)T-\Theta\cdot \ell(\gamma)^2+\frac{L_0}{T}\ell(\gamma)^2.
\end{split} \end{equation}
Since we are assuming $\kappa>e_0(L)$, the above estimate implies that
$T> L_0/\Theta$ and that
\[
\ell(\gamma)^2>\frac{\bigl(\kappa-e_0(L)\bigr)T}{\Theta-L_0/T}>\frac{\bigl(\kappa-e_0(L)\bigr)T}{\Theta}>\frac{\bigl(\kappa-e_0(L)\bigr)L_0}{\Theta^2},
\]
which proves (\ref{claim}).

Fix some number $r$ in the open interval $(0,r_1)$.
Since $z=(x,T)\in \mathcal{Z}_0$, $\SSS_{\kappa}(x(1),T(1))$ is negative, so by (\ref{claim}) the length of $x(1)$ is larger than $r_1$. By continuity, using the fact that $x(0)$ is a constant loop, we get the existence of $\sigma\in(0,1)$ for which $\ell(x(\sigma))=r$. Then (\ref{stma}) implies
\[
\SSS_{\kappa}(x(\sigma),T(\sigma)) \geq \bigl(\kappa-e_0(L)\bigr) T +\Bigr(\frac{L_0}{T}-\Theta\Bigr)r^2.
\]
Minimization in $T$ yields
\[
\SSS_{\kappa}(x(\sigma),T(\sigma)) \geq
r \bigr(\sqrt{L_0(\kappa-e_0(L))}-\Theta r \bigr)=: a.
\]
The number $a$ is positive because $r<r_1$. This concludes the proof.
\end{proof}

\paragraph{\bf The energy range $\mathbf{(\min E,e_0(L))}$} When $\kappa<e_0(L)$, the identity (\ref{constants}) shows that $\SSS_{\kappa}$ takes negative values on some constant loops, and the conclusion of Lemma \ref{Lem:mountain pass} cannot hold. Instead than considering the class of paths which go from some constant loop to a loop of negative action, one has to consider the class of deformations of the space of constant loops - which is diffeomorphic to $M$ - into the space of loops with negative action. More precisely, we consider the set of continuous maps
\[
\mathcal{Z}_M = \bigl\{ (x,T): [0,1]\times M \rightarrow \MM \, \big| \, x(0,x_0) = x_0 \mbox{ and } \SSS_{\kappa}\bigl(x(1,x_0),T(1,x_0)\bigr)<0, \; \forall x_0\in M\bigr\}.
\]

\begin{Lemma}
\label{notempty}
If $\kappa<c_u(L)$, then the set $\mathcal{Z}_M$ is not empty.
\end{Lemma}

We just sketch the proof, referring to \cite{tai83} for more details (see also \cite{tai10}). The argument follows closely  Bangert's technique of ``pulling one loop at a time'' (see \cite{ban80} and \cite{bk83}).

Let $M_0\subset M_1 \subset \dots \subset M_n = M$ be a CW-complex decomposition of $M$. Since $\kappa<c_u(L)$ and since the 0-skeleton $M_0$ is a finite set, it is easy to construct a continuous map
\[
z_0:[0,1]\times M \rightarrow \MM, \quad z_0(\sigma,x_0) = \bigl(y_0(\sigma,x_0),T_0(\sigma,x_0)\bigr),
\]
such that
\begin{enumerate}
\item $y_0(0,x_0) = x_0$ for every $x_0\in M$;
\item $\SSS_{\kappa}\circ z_0(1,x_0)<0$ for every $x_0\in M_0$.
\end{enumerate}
Given a positive integer $h$, we may iterate each loop $h$ times and obtain the map
\[
z_0^h: [0,1]\times M\rightarrow \MM, \quad z_0^h(\sigma,x_0) = \bigl(y_0^h(\sigma,x_0),hT_0(\sigma,x_0)\bigr),
\]
where
\[
y_0^h(\sigma,x_0)(s) := y_0(\sigma,x_0)(hs), \quad \forall (\sigma,x_0)\in [0,1]\times M, \; \forall s\in \T.
\]
Consider an edge $S$ in $M_1$ with end-points $x_0,x_1\in M_0$. The map $z_0^h(1,\cdot)$ maps the the end-points of $S$ into the $h$-th iterates $\alpha^h$ and $\beta^h$ of two loops $\alpha$ and $\beta$ with negative action $\SSS_{\kappa}$. By pulling one of the $h$ loops at a time from $\alpha^h$ to $\beta^h$, one can construct a new map from $S$ into $\MM$ with end-points $\alpha^h$ and $\beta^h$ and such that $\SSS_{\kappa}$ is negative on its image, provided that $h$ is large enough. By relying on the map $z_0^h$, this construction can be done globally, and one ends up with a continuous map
\[
z_1:[0,1]\times M\rightarrow \MM, \quad z_1(\sigma,x_0) = \bigl(y_1(\sigma,x_0),T_1(\sigma,x_0)\bigr),
\]
such that
\renewcommand{\theenumi}{\roman{enumi}}
\renewcommand{\labelenumi}{(\theenumi')}
\begin{enumerate}
\item $y_1(0,x_0) = x_0$ for every $x_0\in M$;
\item $\SSS_{\kappa}\circ z_1(1,x_0)<0$ for every $x_0\in M_1$.
\end{enumerate}
Iterating this process, one can construct continuous maps
\[
z_k:[0,1]\times M\rightarrow \MM, \quad z_k(\sigma,x_0) = \bigl(y_k(\sigma,x_0),T_k(\sigma,x_0)\bigr),
\]
such that
\renewcommand{\theenumi}{\roman{enumi}}
\renewcommand{\labelenumi}{(\theenumi'')}
\begin{enumerate}
\item $y_k(0,x_0) = x_0$ for every $x_0\in M$;
\item $\SSS_{\kappa}\circ z_k(1,x_0)<0$ for every $x_0\in M_k$.
\end{enumerate}
\renewcommand{\theenumi}{\roman{enumi}}
\renewcommand{\labelenumi}{(\theenumi)}
The map $z_n$ is an element of $\mathcal{Z}_M$. This concludes our sketch of the proof of Lemma \ref{notempty}.
The proof of the following result is analogous to the proof of Lemma \ref{Lem:mountain pass}.

\begin{Lemma}
\label{mtpass2}
Assume that $\min E <\kappa < c_u(L)$. Then there exists $a>0$ such that for every $z\in \mathcal{Z}_M$ there holds
\[
\max_{(\sigma,x_0) \in[0,1]\times M} \SSS_{\kappa} \bigl(z(\sigma,x_0)\bigr)\geq a.
\]
\end{Lemma}

\section{Periodic orbits with low energy}
\label{pole}

\paragraph{\bf The Struwe monotonicity argument} When $\kappa\leq c_u(L)$, the periods in a (PS) sequence need not be bounded anymore. Because of this fact, the question of the existence of periodic orbits for every energy $\kappa$ in the interval $[\min E,c_u(L)]$ is open, although no counterexamples are known. 
The known system which is closer to being a counterexample is the horocycle flow on a closed surface $M$ with constant negative curvature (see e.g. \cite{man91, cmp04}): Such a flow has no periodic orbits (actually, every orbit is dense) and it is the restriction of a Hamiltonian flow to an energy surface at a Ma\~{n}\'e critical value, but the corresponding Lagrangian is well defined only on the (non compact) universal cover of $M$ (such a system belongs to the family of {\em non-exact magnetic flows}, whereas only {\em exact} magnetic flows can be described by a Tonelli Lagrangian on $TM$).

The following argument is a version of an argument of Struwe, which says that when dealing with a minimax value associated to a family of functionals depending on a real parameter in a suitable monotone way, there exist compact (PS) sequences for almost every value of the parameter. This argument has applications both to Hamiltonian periodic orbits and to semilinear elliptic equations
(see \cite{str90}, \cite[section II.9]{str00} and references therein).

Let us assume that $\min E<c_u(L)$, otherwise the interval of low energies collapses to a single level and there is nothing to prove.
Given $\kappa\in (\min E,c_u(L))$, let $\Gamma$ be the set of the images of the maps either in $\mathcal{Z}_0$ or in $\mathcal{Z}_M$, which were introduced in the previous section: If $e_0(L)< \kappa < c_u(L)$ we may take $\mathcal{Z}_0$, while in general we should take $\mathcal{Z}_M$. Let $I$ be either the interval $(e_0(L),c_u(L))$ - if we are dealing with $\mathcal{Z}_0$ - or the interval $(\min E,c_u(L))$ - if we are dealing with $\mathcal{Z}_M$. For every $\kappa\in I$, consider the minimax value
\begin{equation}
\label{minimax}
c(\kappa) := \inf_{K \in \Gamma} \max_{(x,T)\in K} \SSS_{\kappa} (x,T).
\end{equation}
By Lemmas \ref{Lem:mountain pass}, \ref{notempty}, and \ref{mtpass2}, $c(\kappa)$ is finite and positive, and since $\SSS_{\kappa}$ depends monotonically on $\kappa$, the function
\[
c:I \rightarrow (0,+\infty)
\]
is weakly increasing. By Lebesgue Theorem, the set of points of $I$ at which $c$ has a linear modulus of continuity, that is
\[
J := \bigl\{ \bar\kappa\in I\, \big| \, \exists \delta>0, \; \exists M>0 \mbox{ s.t. } |c(\kappa) - c(\bar{\kappa})| \leq M |\kappa - \bar{\kappa}| \mbox{ for every } \kappa \in I \mbox{ with }|\kappa - \bar\kappa| < \delta\bigr\},
\]
has full Lebesgue measure in $I$.

\begin{Lemma}
\label{strarg}
If $\bar \kappa\in J$, then $\SSS_{\bar\kappa}$ admits a bounded (PS) sequence at level $c(\bar\kappa)$, which consists of contractible loops.
\end{Lemma}

\begin{proof}
First recall that $\Gamma$ is a class of subsets of $\MM^{\mathrm{contr}}$.
Let $(\kappa_h)\subset I$ be a strictly decreasing sequence which converges to $\bar\kappa$, and set $\epsilon_h:=\kappa_h-\bar \kappa\downarrow0$. We pick $K_h \in \Gamma$ such that
\[
\max_{K_h}\SSS_{\kappa_h}\leq c(\kappa_h)+\epsilon_h.
\]
Let $z=(x,T)\in K_h$ be such that $\SSS_{\bar \kappa}(z)>c(\bar \kappa)-\epsilon_h$. Since $\bar\kappa$ belongs to $J$, we have
\[
T = \frac{\SSS_{\kappa_h}(z)-\SSS_{\bar \kappa}(z)}{\kappa_h-\bar\kappa} \leq \frac{c(\kappa_h)+\epsilon_h-c(\bar\kappa)+\epsilon_h}{\epsilon_h}\leq M+2.
\]
Moreover,
\[
\SSS_{\bar \kappa}(z) \leq \SSS_{\kappa_h}(z) \leq c(\kappa_h)+\epsilon_h \leq c(\bar \kappa)+(M+1)\epsilon_h.
\]
By the above considerations,
\[
K_h \subset A_h \cup\big\{\SSS_{\bar \kappa}\leq c(\bar \kappa)-\epsilon_h\big\},
\]
where
\[
A_h := \big\{(x,T)\,\big|\, T\leq M+2\mbox{ and }\SSS_{\bar \kappa} (x,T) \leq c(\bar \kappa)+(M+1)\epsilon_h\big\}.
\]
If $(x,T)$ belongs to $A_h$, we have the estimate
\[
\SSS_{\bar\kappa}(x,T)  \geq \frac{L_0}{M+2}\|x'\|_2^2-(M+2)(L_1-\bar\kappa),
\]
(see (\ref{isbd})), which shows that $A_h$ is bounded in $\MM$, uniformly in $h$. Let $\phi$ be the flow of the vector field obtained by multiplying $-\nabla\SSS_{\bar \kappa}$ by a suitable non-negative function, whose role is to make the vector field bounded on $\MM$ and vanishing on the sublevel $\{\SSS_{\bar\kappa} \leq c(\bar\kappa)/4\}$, while keeping the uniform decrease condition
\begin{equation}
\label{ud}
\frac{d}{d\sigma} \SSS_{\bar\kappa} \bigl(\phi_{\sigma}(z)\bigr) \leq - \frac{1}{2} \min \bigl\{
\|d\SSS_{\bar\kappa} (\phi_{\sigma}(z)) \|^2,1 \bigr\}, \quad \mbox{if } \SSS_{\bar\kappa} (\phi_{\sigma}(z)) \geq c(\bar\kappa)/2.
\end{equation}
See (\ref{decr}) and Remarks \ref{noncomp}, \ref{trunc}. Then Lemma \ref{Lem:4} implies that $\phi$ is well-defined on $[0,+\infty[\times \MM$, and the class of sets $\Gamma$ is positively invariant with respect to $\phi$. Since $\phi$ maps bounded sets into bounded sets, we have
\begin{equation}
\label{dove}
\phi([0,1]\times K_h) \subset B_h \cup \bigl\{\SSS_{\bar \kappa}\leq c(\bar \kappa)-\epsilon_h \bigr\},
\end{equation}
for some uniformly bounded set
\begin{equation}
\label{bdd}
B_h\subset \bigl\{\SSS_{\bar \kappa}\leq c(\bar \kappa)+(M+1)\epsilon_h \bigr\}.
\end{equation}
We claim that there exists a sequence $(z_h)\subset \MM^{\mathrm{contr}}$ with
\[
z_h \in B_h\cap \bigl\{\SSS_{\bar \kappa}\geq c(\bar \kappa)-\epsilon_h \bigr\},
\]
and $\| d\SSS_{\bar \kappa}(z_h)\|$ infinitesimal. Such a sequence is clearly a bounded (PS) sequence at level $c(\bar\kappa)$.
Assume, by contradiction, the above claim to be false. Then there exists $0<\delta<1$ which satisfies
\[
\|d\SSS_{\bar \kappa}\| \geq \delta \quad \mbox{on } B_h \cap \bigl\{\SSS_{\bar \kappa}\geq c(\bar \kappa)-\epsilon_h\bigr\},
\]
for every $h$ large enough. Together with (\ref{ud}), (\ref{dove}) and (\ref{bdd}), this implies that, for $h$ large enough, for any $z\in K_h$ such that
\[
\phi([0,1]\times \{z\}) \subset \bigl\{ \SSS_{\kappa} \geq c(\bar \kappa)-\epsilon_h\bigr\},
\]
there holds
\[
\SSS_{\bar \kappa} \big(\phi_1(z)\big) \leq \SSS_{\bar \kappa}(z)- \frac{1}{2} \delta^2 \leq c(\bar \kappa)+ (M+1) \epsilon_h- \frac{1}{2} \delta^2.
\]
It follows that
\[
\max_{\phi_1(K_h)} \SSS_{\bar \kappa} \leq c(\bar \kappa)-\epsilon_h,
\]
for $h$ large enough. Since $\phi_1(K_h)$ belongs to $\Gamma$, this contradicts the definition of $c(\bar\kappa)$ and concludes the proof.
\end{proof}

\paragraph{\bf Existence of periodic orbits of low energy} We are finally ready to prove the following result, which is statement (iii) in the theorem of the Introduction:

\begin{Thm}
\label{ae}
For almost every $\kappa \in (\min E,c_u(L))$, there is a contractible periodic orbit $\gamma$ of energy $E(\gamma,\gamma')=\kappa$ and positive action $\SSS_{\kappa}(\gamma)=c(\kappa)$.
\end{Thm}

\begin{proof}
Let $\kappa$ be an element of the full measure set $J\subset I$. 
By Lemma \ref{strarg}, $\SSS_{\kappa}$ admits a (PS)$_{c(\kappa)}$ sequence $(x_h,T_h)\subset \MM^{\mathrm{contr}}$ with $(T_h)$ bounded. By Lemma \ref{PST0}, $(T_h)$ is bounded away from zero, because $c(\kappa)>0$.
By Lemma \ref{Lem;2}, the sequence $(x_h,T_h)$ has a limiting point in $\MM^{\mathrm{contr}}$, which gives us the required contractible periodic orbit.
\end{proof}

\begin{Rmk}
The existence of a periodic orbit for almost energy level in $(\min E,e_0(L))$ can be proved also by an argument from symplectic topology. In fact, let $H:T^*M \rightarrow \R$ be the Hamiltonian which is Legendre dual to $L$.
The fact that $\kappa< e_0(L)$ implies that the restriction of the projection $T^*M\rightarrow M$ to $H^{-1}(\kappa)$ is not surjective. Therefore, one can build a Hamiltonian diffeomorphism of $T^*M$ which displaces $H^{-1}(\kappa)$ from itself (see \cite[Proposition 8.2]{con06}). Sets which are displaceable by a Hamiltonian diffeomorphism have finite $\pi_1$-sensitive Hofer-Zehnder capacity (see \cite{sch06} and \cite{fs07}) and this fact implies the almost everywhere existence result for periodic orbits (see \cite{hz94}). See \cite[Corollary 8.3]{con06} for more details on such a proof.
\end{Rmk}

The next result shows that stable energy levels of Tonelli Hamiltonians posses periodic orbits, proving statement (iv) of the theorem in the Introduction. In particular, the same is true for contact type energy levels.

\begin{Cor}
Assume that $\kappa$ is a regular value of the Tonelli Hamiltonian $H\in C^{\infty}(T^*M)$ and that the hypersurface
$\Sigma=H^{-1}(\kappa)$ is stable. Then $\Sigma$ carries a periodic orbit.
\end{Cor}

\begin{proof}
By stability, we can find a diffeomorphism 
\[
(-\epsilon, \epsilon) \times \Sigma \rightarrow T^*M, \qquad (r,x) \mapsto \psi_r(x),
\]
onto an open neighborhood of $\Sigma$ such that $\psi_0$ is the identity on $\Sigma$ and 
\[
\psi_r : \Sigma \rightarrow \Sigma_r := \psi_r(\Sigma)
\]
induces an isomorphism between the line bundles $\mathcal{L}_{\Sigma}$ and $\mathcal{L}_{\Sigma_r}$. Up to the choice of a smaller $\epsilon$, we may assume that all the hypersurfaces $\Sigma_r$ are levels of a uniformly convex function. Therefore, they are the level sets of a Tonelli Hamiltonian $K\in C^{\infty}(T^* M )$ (see \cite{mp10} for a detailed construction of $K$). Since the Legendre transform of $K$ is a Tonelli Lagrangian, Theorems \ref{highthm} and \ref{ae} imply that $K^{-1}(\kappa)$ has periodic orbits for almost every $\kappa$. In particular, $\Sigma_r$ has periodic orbits for almost every $r$, but since the dynamics on $\Sigma_r$ and on $\Sigma$ are conjugated, the same is true for $\Sigma$. 
\end{proof}

\begin{Rmk}
The above proof shows the usefulness of having a theory which works with Tonelli Lagrangians, rather than just electromagnetic ones. In fact even if the stable hypersurface $\Sigma$ is the level set of an electromagnetic Hamiltonian (that is, it is fiberwise an ellipsoid), the hypersurfaces $\Sigma_r$ given by the stability assumption may be more general fiberwise uniformly convex hypersurfaces. 
\end{Rmk}

\begin{Rmk}
If $E^{-1}(\kappa)$ is of contact type and $\pi_* : H_1(E^{-1}(\kappa),\R) \rightarrow H_1(M,\R)$ is injective, then $\SSS_{\kappa}$ satisfies the Palais-Smale condition (with a suitable choice of the metric of $\mathcal{M}$). See \cite[Proposition F]{con06}. Therefore, in this case the existence of a periodic orbit can be obtained also without using stability.
\end{Rmk}

\begin{Rmk}
It can be proved that when $M$ is a closed surface and $L$ is of the form (\ref{elmag}) with $V=0$ (that is, in the case of exact magnetic flows), there are periodic orbits on {\em every} energy level below $c_0(L)$ (see  \cite{tai92b}, \cite{tai92c}, \cite{tai92} and \cite{cmp04}). In fact, the advantage of dealing with a surface is that when $\kappa<c_0(L)$ one can minimize $\SSS_{\kappa}$ on a suitable space of {\em embedded} closed curves. In the same setting, one can prove that for almost every energy level below $c_u(L)$ there are infinitely many periodic orbits, at least if all periodic orbits are assumed to be non-degenerate (see \cite{amp13}).
\end{Rmk}

\paragraph{\textbf{The two Lyapunov functions argument}} We conclude these notes by discussing an alternative argument to deal with the lack of (PS) which is exhibited by $\SSS_{\kappa}$ when $\kappa<c_u(L)$. It allows to prove that the set of energy levels $\kappa$ such that the Euler-Lagrange flow has a periodic orbit of energy $\kappa$ is {\em dense} in $(\min E,c_u(L))$, a weaker statement than Theorem \ref{ae}. However, it has some advantages, which are discussed in Remark \ref{adv} below. This argument is used, in a different context, in \cite{ama08}. Here we shall use it in order to prove the following weaker version of Theorem \ref{ae}:

\begin{Thm}
\label{dense}
Let $\min E<\bar \kappa<c_u(L)$ and assume that there are no contractible periodic orbits of energy $\bar\kappa$ and non-negative action $\SSS_{\bar\kappa}$.
Then there exists a strictly decreasing sequence $(\kappa_h)$ which converges to $\bar\kappa$ and is such that the Euler-Lagrange flow has a contractible periodic orbit $\gamma_h$ with energy $\kappa_h$ and period $T_h$, which satisfies $\SSS_{\kappa_h}(\gamma_h)/T_h \downarrow 0$.
\end{Thm}

\begin{proof}
We argue by contradiction and  assume that there exists $\tilde\kappa>\bar\kappa$ and $\delta>0$ such that for any $\kappa \in[\bar \kappa,\tilde \kappa]$ all the periodic orbits $\gamma$ of
energy $\kappa$ and period $T$ satisfy either $\SSS_{\kappa}(\gamma)/T\geq\delta$ or $\SSS_{\kappa} (\gamma)\leq 0$. Fix real numbers $a>c(\bar \kappa)$ and  $\kappa^*\in(\bar \kappa,\tilde \kappa]$. Assume that we can find $\lambda\in[0,1]$ and $(x,T)\in\MM$ such that
\[
\lambda \, d\SSS_{\bar \kappa}(x,T)+(1-\lambda)\, d\SSS_{\kappa^*}(x,T)=0, \quad 0<\SSS_{\bar \kappa}(x,T)\leq a.
\]
Then $(x,T)$ is a critical point of $\SSS_{\lambda\bar \kappa+(1-\lambda)\kappa^*}$, hence it is a $T$-periodic orbit with energy $\lambda\bar \kappa+(1-\lambda)\kappa^*$. By what we have assumed at the beginning, we have
\[
\delta \leq \frac{1}{T}\SSS_{\lambda\bar \kappa+(1-\lambda)\kappa^*}(x,T)
= \frac{1}{T} \SSS_{\bar \kappa}(x,T)+(1-\lambda)(\kappa^*-\bar \kappa)
\leq \frac{a}{T} + \kappa^*-\bar\kappa.
\]
Up to the choice of a smaller $\kappa^*>\bar\kappa$, we may assume that $\kappa^*-\bar \kappa\leq \delta/2$. Then the above estimate implies that
\[
T\leq \frac{2a}{\delta} =: T^*.
\]
With these choices of $\kappa^*$ and $T^*$, we can restate what we have proved so far as:

\begin{Lemma}
\label{oppo}
If $T>T^*$ and $0<\SSS_{\bar \kappa}(x,T)\leq a$, then the segment
\[
\mathrm{conv} \bigl\{ d\SSS_{\bar \kappa}(x,T), d\SSS_{\kappa^*}(x,T) \bigr\} \subset T_{(x,T)}^* \MM
\]
does not contain $0$.
\end{Lemma}

The above lemma allows us to construct a negative pseudo-gradient vector field for $\SSS_{\bar\kappa}$ which has all the good properties of $-\nabla \SSS_{\bar\kappa}$ and moreover has $\SSS_{\kappa^*}$ as a Lyapunov function on the open set
\[
A:= \bigl\{T>T^*\} \cap \{0< \SSS_{\bar\kappa} < a \bigr\}.
\]
In fact, the only obstruction to finding a vector field $W$ whose flow make both $\SSS_{\bar\kappa}$ and $\SSS_{\kappa^*}$ decrease in $A$, is that the differentials of $\SSS_{\bar\kappa}$ and $\SSS_{\kappa^*}$ point in opposite directions in some point of $A$, and this is precisely what is excluded by Lemma \ref{oppo}.
More precisely, one can prove the following:

\begin{Lemma}
\label{pg}
There exists a locally Lipschitz vector field $W$ on $\MM$ such that:
\begin{enumerate}
\item $d\SSS_{\bar\kappa}[W]<0$ on $\{\SSS_{\bar\kappa}>0\}$;
\item $W$ is forward complete and vanishes on $\{\SSS_{\bar\kappa}\leq 0\}$;
\item let $z_h=(x_h,T_h)$ be a sequence in $\MM^{\mathrm{contr}}$ such that
\[
0< \inf \SSS_{\bar{\kappa}}(z_h) \leq \sup \SSS_{\bar{\kappa}}(z_h) < +\infty, \quad \lim_{h\rightarrow \infty} d\SSS_{\bar\kappa}(z_h)[W(z_h)] = 0,
\]
and $(T_h)$ is bounded from above; then $(z_h)$ has a subsequence which converges in $\MM^{\mathrm{contr}}$;
\item $d\SSS_{\kappa^*}[W]<0$ on $A$.
\end{enumerate}
\end{Lemma}

In fact, one can choose $W$ to be given by the vector field
\[
\nabla\SSS_{\bar \kappa}+\chi \frac{\|\nabla\SSS_{\bar \kappa}\|}{\|\nabla\SSS_{\kappa^*}\|}\nabla\SSS_{\kappa^*}
\]
multiplied by a suitable non-positive function. Here $\chi$ is a suitable cut-off function. See \cite[Lemmas 5.1 and 5.4]{ama08} for a similar construction.

We can now prove Theorem \ref{dense}.
By the definition of $c(\bar\kappa)$, there is a set $K$ in $\Gamma$ such that
\[
\max_{K} \SSS_{\bar \kappa} <a.
\]
By Lemma \ref{pg} (i) and (ii), for every $\sigma_0>0$ we have
\[
\inf_{\sigma\in \sigma_0} \Bigl| d\SSS_{\bar\kappa}\bigl(\phi_{\sigma}(z)\bigr)\bigl[W(\phi_{\sigma}(z))\bigr] \Bigr| \leq \frac{1}{\sigma_0} \int_0^{\sigma_0} \Bigl| d\SSS_{\bar\kappa}\bigl(\phi_{\sigma}(z)\bigr)\bigl[W(\phi_{\sigma}(z))\bigr]\Bigr| \, d\sigma = \frac{
\SSS_{\bar\kappa}(z) - \SSS_{\bar\kappa}(\phi_{\sigma_0}(z))}{\sigma_0},
\]
and, by the definition of $c(\bar\kappa)$,
\[
\max_{z\in K} \SSS_{\bar\kappa} \bigl( \phi_{\sigma_0}(z) \bigr) \geq c(\bar\kappa).
\]
By taking a limit for $\sigma_0\rightarrow +\infty$, thanks to Lemma \ref{pg} (ii), the above facts imply that $\phi_{\R^+}(K)\cap \{\SSS_{\bar\kappa}>0\}$ contains a sequence $z_h=(x_h,T_h)$ such that
\[
0< c(\bar\kappa) \leq \inf \SSS_{\bar{\kappa}}(z_h) \leq \sup \SSS_{\bar{\kappa}}(z_h) < a \quad \mbox{and} \quad \lim_{h\rightarrow \infty} d\SSS_{\bar\kappa}(z_h)[W(z_h)] = 0.
\]
It is enough to show that $(T_h)$ is bounded from above: Indeed, in this case Lemma \ref{pg} (iii) implies that $(z_h)$ has a limiting point, which is a critical point of $\SSS_{\bar\kappa}$ with positive action, contradicting the hypothesis of Theorem \ref{dense}.

The upper bound on $(T_h)$ is a consequence of the following claim: the period $T$ is bounded on the set $\phi_{\R^+}(K)\cap \{\SSS_{\bar\kappa}>0\}$.
In order to prove this claim, we first notice that
\begin{equation}
\label{hhh}
\SSS_{\bar \kappa}(x,T)\leq a, \;\; T\leq T^* \quad \Rightarrow \quad
\SSS_{\kappa^*}(x,T)\leq a+(\kappa^*-\bar \kappa)T^*=:b.
\end{equation}
Since $K$ is compact, we can find $d>b$ such that $K\subset \{\SSS_{\kappa^*} < d\}$. Let $\phi$ be the flow of the vector field $W$. We claim that
\begin{equation}
\label{cc}
\phi_{\R^+}(K) \cap \bigl\{\SSS_{\bar\kappa} > 0\bigr\} \subset \bigl\{\SSS_{\kappa^*} < d\bigr\}.
\end{equation}
In fact, let $z\in K$ and let $\sigma_0>0$ be the first instant such that $\SSS_{\kappa^*}(\phi_{\sigma_0}(z))=d$, while $\SSS_{\bar{\kappa}}(\phi_{\sigma_0}(z))> 0$. By Lemma \ref{pg} (i), $\SSS_{\bar\kappa}(\phi_{\sigma_0}(z)) \leq \SSS_{\bar\kappa}(z) < a$. By Lemma \ref{pg} (iv), the point $\phi_{\sigma_0}(z)$ cannot belong to $A$, so $\phi_{\sigma_0}(z)=(x,T)$ with $T\leq T^*$ and (\ref{hhh}) implies that $\SSS_{\kappa^*}(\phi_{\sigma_0}(z)) \leq b < d$. This contradiction proves (\ref{cc}).

If $\SSS_{\bar\kappa}(x,T)>0$ and $\SSS_{\kappa^*}(x,T)<d$, then
\[
d >  \SSS_{\kappa^*}(x,T) = \SSS_{\bar\kappa}(x,T) + (\kappa^*-\bar\kappa) T > (\kappa^*-\bar\kappa) T.
\]
This shows that the period $T$ is bounded on the set
\[
\bigl\{\SSS_{\bar\kappa} > 0\bigr\} \cap \bigl\{\SSS_{\kappa^*} < d\bigr\},
\]
and by (\ref{cc}) it is bounded also on
\[
\phi_{\R^+}(K) \cap \bigl\{\SSS_{\bar\kappa}>0\},
\]
as claimed.
\end{proof}

\begin{Rmk}
\label{adv}
In the Struwe monotonicity argument, one gets the existence of bounded (PS) sequences at level $c(\bar{\kappa})$, but has no control on the (PS) sequences at other levels. Therefore, it is not clear whether the space of negative gradient flow lines for $\SSS_{\bar\kappa}$ which connect two given critical points - say with positive action - is bounded.
An advantage of the two Lyapunov functions argument, is that the latter fact is true for the flow lines of the vector field $W$ constructed in Lemma \ref{pg}: the second Lyapunov function $\SSS_{\kappa^*}$ allows to exclude the existence of flow lines which go arbitrarily far and come back. This fact would allow to develop some global critical point theory for $\SSS_{\bar\kappa}$, such as Morse theory or Lusternik-Schnirelmann theory. This is not useful here, because the a priori estimates which lead to the existence of the pseudo-gradient vector field $W$ come from a contradiction argument. However, it might be useful in situations where these a priori bounds have a different origin, such as for example in the case of tame energy levels (see \cite{cfp10} for the definition of tameness and for motivating examples).
\end{Rmk}


\providecommand{\bysame}{\leavevmode\hbox to3em{\hrulefill}\thinspace}
\providecommand{\MR}{\relax\ifhmode\unskip\space\fi MR }
\providecommand{\MRhref}[2]{%
  \href{http://www.ams.org/mathscinet-getitem?mr=#1}{#2}
}
\providecommand{\href}[2]{#2}

\end{document}